\documentclass{article}

\usepackage[english]{babel}
\usepackage[utf8]{inputenc}
\usepackage{amsmath}
\usepackage{caption}
\usepackage{subcaption}
\usepackage{amssymb}
\usepackage[normalem]{ulem}
\usepackage{amsthm}
\usepackage{comment}
\usepackage{faktor}
\usepackage{cancel}
\usepackage[toc,page]{appendix}
\usepackage{hyperref}
\hypersetup{%
	pdfpagemode={UseOutlines},
	bookmarksopen,
	pdfstartview={FitH},
	linkcolor={black},
	citecolor={black},
	urlcolor={black}
}
\usepackage{amsfonts}
\usepackage{mathrsfs}
\usepackage{physics}
\usepackage{graphicx}
\usepackage{xcolor}
\usepackage[colorinlistoftodos]{todonotes}
\numberwithin{equation}{section}
\newtheorem{theorem}{Theorem}[section]
\newtheorem{corollary}[theorem]{Corollary}
\newtheorem{lemma}[theorem]{Lemma}
\newtheorem{prop}[theorem]{Proposition}
\newtheorem{deff}[theorem]{Definition}
\newtheorem{remark}[theorem]{Remark}
\theoremstyle{definition}

\newtheorem{example}[theorem]{Example}

\newcommand{\gen}{\operatorname{Gen}}

\newcommand{\lpl}{\left\langle}
\newcommand{\rer}{\right\rangle}

\renewcommand{\theta}{\vartheta}

\newcommand{\eps}{\varepsilon}
\newcommand{\mc}{\mathcal}

\newcommand{\mbf}{\mathbf}

\renewcommand{\l}{\left(}
\renewcommand{\r}{\right)}

\renewcommand{\lll}{\left\{}
\newcommand{\rrr}{\right\}}

\newcommand{\orr}{\operatorname{Orr}}
\newcommand{\ray}{\operatorname{Ray}}

\usepackage{arxiv}

\usepackage[utf8]{inputenc} 
\usepackage[T1]{fontenc}    
\usepackage{hyperref}       
\usepackage{url}            
\usepackage{booktabs}       
\usepackage{amsfonts}       
\usepackage{nicefrac}       
\usepackage{microtype}      
\usepackage{lipsum}		
\usepackage{graphicx}
\usepackage{doi}

\title{Full instability of boundary layers with the Navier boundary condition}
\date{}

\author{{Lorenzo Quarisa}\thanks{LQ is supported by an EU Chancellor's scholarship from the University of Warwick.\\
For the purpose of open access, the authors have applied a Creative Commons Attribution (CC-BY) license to Any Author Accepted Manuscript version arising from this submission.} \\
	Mathematics Institute\\
	University of Warwick\\
	Coventry CV47AL, United Kingdom \\
	\texttt{lorenzo.quarisa@warwick.ac.uk}  \\
	\And
{José L. Rodrigo} \\
	Mathematics Institute\\
	University of Warwick\\
	Coventry CV47AL, United Kingdom \\
	\texttt{j.rodrigo@warwick.ac.uk} \\
}

\renewcommand{\shorttitle}{Instability of boundary layers with the Navier boundary condition}

\hypersetup{
pdftitle={Instability of boundary layers with the Navier boundary condition},
pdfauthor={David S.~Hippocampus, Elias D.~Striatum},
pdfkeywords={Boundary layers, Prandtl expansion, zero viscosity limit, Navier friction boundary condition},
}

\begin{document}
   
 \today
    
\maketitle

\begin{abstract}
	 We consider the problem of the stability of the Navier-Stokes equations in $\mathbb{T}\times \mathbb{R}_+$ near shear flows which are linearly unstable for the Euler equation. In \cite{greniernguyen}, the authors prove an $L^{\infty}$ instability result for the no-slip boundary condition which also denies the validity of the Prandtl boundary layer expansion. In this paper, we generalise this result to a Navier slip boundary condition with viscosity dependent slip length: $\partial_y u =\nu^{-\gamma}u$ at $y=0$, where $\gamma >1/2$. This range includes the physical slip rate $\gamma=1$. 
\end{abstract}


\tableofcontents 

\section{Introduction}
We consider the following problem for the Navier-Stokes equations in $\mathbb{T}\times \mathbb{R}_+$. Let $\mbf{u}=(u(t,x,y),v(t,x,y))$ be a smooth flow. Then the Navier-Stokes equations with a viscosity-dependent Navier boundary condition and forcing $\mbf{f}$ may be written as
\begin{equation}\label{eq:ns}\begin{cases}
\partial_t \mbf{u}+(\mbf{u}\cdot \nabla)\mbf{u}+\nabla p = {\nu}\Delta \mbf{u}+\mbf{f} & \text{on }\mathbb{R}^2_+;\\ 
\nabla \cdot \mbf{u}=0 &\text{on } \mathbb{R}^2_+;\\ 
\partial_y u = \nu^{-\gamma} u & y=0;\\ 
v = 0 & y=0.\end{cases}
\end{equation}
Let $\mbf{U}_0 = (U_0(y),0)$ be a smooth shear flow (we will also refer to its first component $U_0$ as \emph{shear flow}). We can then consider the linearized Euler equations around $\mbf{U}_0$:
\begin{equation}\label{eq:eulinear}\begin{cases}
\partial_t \mbf{u}+(\mbf{u}\cdot \nabla)\mbf{U}_0 + (\mbf{U}_0\cdot \nabla ) \mbf{u}+\nabla p =0 & \text{on }\mathbb{R}^2_+;\\ 
\nabla \cdot \mbf{u}=0 &\text{on } \mathbb{R}^2_+;\\ 
v = 0 & y=0.\end{cases}
\end{equation}
The shear flow $\mbf{U}_0$ is called \emph{linearly unstable} for the Euler equations when \eqref{eq:eulinear} admits a solution which is exponentially growing in time.

With this boundary condition, the Prandtl expansion, which states that solutions to \eqref{eq:ns} may be approximated by the sum of an inviscid solution to the Euler equations and a \emph{boundary layer} term,
\begin{equation}\label{eq:expgen}\mbf{u}^{\nu}(t,x,y) \sim \mbf{u}^E(t,x,y)+ \nu^{\max\lll 1/2-\gamma;0\rrr} \mbf{u}^b \l t,x,\frac{y}{\sqrt{\nu}}\r,\qquad \text{ as }\nu\to 0^+, \end{equation}

is expected to display stronger stability properties. For instance, Iftimie and Sueur \cite{iftimie} proved that the above expansion is valid in energy spaces when $\gamma=0$ (viscosity-independent case). More recently in \cite{tao}, the authors exhibit an analogous result for $\gamma\in (0,1/2]$ with initial condition in the Gevrey class $(2\gamma)^{-1}$, which is simply the analytic class in the case $\gamma=1/2$.

The goal of this paper is to extend the nonlinear $L^{\infty}$ instability result by Grenier and Nguyen \cite{greniernguyen} near a shear flow $\mbf{U}_0$ which is linearly unstable for the Euler equations, to the case of the Navier boundary condition \eqref{eq:ns}$_3$-\eqref{eq:ns}$_4$. This leads to order one instability of the Prandtl boundary layer expansion for all $\gamma \geq 1/2$. This improves and extends the instability result by Paddick \cite{paddick} for $\gamma=1/2$, as well as the authors' previous result from \cite{prevpaper} for $\gamma \geq 1/2$. This range is particularly interesting as it includes the physically relevant case $\gamma=1$, which corresponds to the slip rate predicted by Navier \cite{navier}.

As the boundary condition depends on the viscosity through a power of $\nu$, we must consider a viscosity-dependent family of shear flows
\begin{equation}\label{eq:shflows}
\begin{cases}\partial_t U^{\nu}(t,y)=\partial_{yy}U^{\nu}(t,y) & t\geq 0,y\geq 0;\\
\partial_{y} U^{\nu}(t,0)=\nu^{1/2-\gamma} U^{\nu}(t,0) & t\geq 0;\\
U^{\nu}(0,y)= U_0(y) & y\geq 0.
\end{cases}
\end{equation}
Note that $y\mapsto \mbf{U}^{\nu}(t,y/\sqrt{\nu}):=(U^{\nu}(t,y/\sqrt{\nu}),0)$ then satisfies the Navier-Stokes equations \eqref{eq:ns} with $\mbf{f}=0$.

We will assume the following conditions on the linearly unstable shear flow $U_0$. It must be smooth, and for some $\eta_{\infty}>0$ it must satisfy:
\begin{align}
 \label{eq:uscond}  \begin{cases} |U_0^{(k)}(y)|\leq C_k e^{-\eta_{\infty}y}& k\geq 1;\\
U_0^{(k)}(0)=0 &k\geq 0;\\
 \sup_{y>0}\frac{y^k}{(1+y)^k}|U_0^{(k)}(y)|\leq C^{k+1} k!& k\geq 0, \text{for some }C>0.
\end{cases}
\end{align}
The second condition \eqref{eq:uscond}$_2$ is required in order for the map $\nu \mapsto U^{\nu}(\sqrt{\nu}t,y)-U_0(y)$ to be continuous in $H^s(\mathbb{R}_+)$ for $t>0$, and corresponds to satisfying the compatibility conditions at all orders at $(t,y)=(0,0)$. It is stronger than the assumption needed in the no-slip case, due to the dependence on the viscosity of the boundary condition (see the Appendix). The third condition \eqref{eq:uscond}$_3$ replaces analyticity, thanks to the weight $y^k(1+y)^{-k}$ which allows the flow to be non-analytic at the origin.

This is our main result.
\begin{theorem}\label{main}Let $\gamma>1/2$. For $\nu >0$, let $\mbf{U}^{\nu}=(U^{\nu},0)\in C^{\infty}(\mathbb{R}_+)$ be a family of shear flows constructed in \eqref{eq:shflows} and with $U_0$ linearly unstable for Euler and satisfying \eqref{eq:uscond}. Then for any $s\geq 0,N_0\in \mathbb{Z}_{\geq 1}$ there exists a family of solutions $\mbf{u}^{\nu}=(u^{\nu},v^{\nu})$ to \eqref{eq:ns}, with a forcing $\mbf{f}^{\nu}$ and a
constant $\sigma>0$ and times $\tilde{T}^{\nu}\searrow 0$ such that for all $\nu>0$,
\begin{align}\label{eq:main0}\|\mathbf{u}^{\nu}(0,x,y)-\mbf{U}_0\l y/\sqrt{\nu}\r\|_{H^s} &\leq \nu^{N_0};\\ \label{eq:forcings}
\|\mbf{f}^{\nu}\|_{L^{\infty}([0,T^{\nu}];H^s)}\leq \nu^{N_0}; \\  \label{eq:main1}
\|\mbf{u}^{\nu}(\tilde{T}^{\nu},x,y)-\mbf{U}^{\nu}(\tilde{T}^{\nu},y/\sqrt{\nu})\|_{L^{\infty}}&\geq \sigma.
\end{align}
\end{theorem}
The result only denies the validity of the Prandtl expansion \eqref{eq:expgen} for $\gamma> 1/2$, so there is no contradiction with \cite{iftimie}. It is an improvement to the main result from \cite{prevpaper}, where the instability in \eqref{eq:main1} was proven with an additional factor of $\nu^{\theta}$, for some $0\leq \theta\leq 1/4$ depending on the value of $\gamma$.

The proof of the result follows the same scheme as \cite{greniernguyen}. We start by applying the isotropic rescaling
\begin{equation}\label{eq:iso}(t,x,y)\mapsto \frac{1}{\sqrt{\nu}}(t,x,y),\end{equation}
where we keep the notation from the previous variables. As a result, \eqref{eq:ns} is transformed into (replacing $\sqrt{\nu}\mbf{f}$ with $\mbf{f}$ for simplicity of notation):
\begin{equation}\label{eq:ns2}\begin{cases}
\partial_t \mbf{u}+(\mbf{u}\cdot \nabla)\mbf{u}+\nabla p = {\sqrt{\nu}}\Delta \mbf{u}+\mbf{f} & \text{on }\mathbb{R}^2_+;\\ 
\nabla \cdot \mbf{u}=0 &\text{on } \mathbb{R}^2_+;\\ 
\partial_y u = \nu^{1/2-\gamma} u & y=0;\\ 
v = 0 & y=0.\end{cases}
\end{equation}

We construct the solutions $\mbf{u}^{\nu}$ to \eqref{eq:ns2} as a sum in the form 
\begin{equation}\label{eq:ansatz0}\mbf{u}^{\nu}(t,x,y)=\mbf{U}^{\nu}(\sqrt{\nu}t,y)+\sum_{n\geq 1}\nu^{nN_s}e^{n\Im c_{1} t} \mbf{u}^n(t,x,y),\qquad N_s=N_0+\max\lll 0;3s/4-7/8\rrr.\end{equation}
This construction implies that at time $t=0$, in the original coordinates we have $\|\mbf{u}^{\nu}(0,x,y)-\mbf{U}_0(y/\sqrt{\nu})\|_{H^s}\sim \nu^{N_0}$ (see Section \ref{instab} for the details, including the derivation of $N_s$). It follows the same principles of the construction used in \cite{prevpaper} and originally conceived in \cite{grenier}. In these papers, one truncates the series in \eqref{eq:ansatz0}, obtaining an approximate solution $\mbf{u}^{\textup{app}}$, considering the remaining terms as a forcing $\mbf{R}^{\textup{app}}.$ The forcing satisfies
\begin{equation}\label{eq:rapp}\|\mbf{R}^{\textup{app}}(t)\|_{H^s}\lesssim \nu^{\theta\l 1+ \frac{M+1}{2^{\Lambda}N_0}\r},\end{equation}
for $t\leq \tilde{T}^{\nu}$ and for some $\theta \in [0,1/4]$ depending only on $\gamma$ (see \cite{prevpaper} for the exact definition). Finally, using Gronwall's lemma one can show that there exists an exact solution $\mbf{u}^{\nu}$ with $\|\mbf{u}^{\nu}(t)-\mbf{u}^{\textup{app}}(t)\|_{H^s}\leq \nu^{\theta}$ for all $t\leq \tilde{T}^{\nu}$, which produces the instability. Whilst the instability produced by \autoref{main} acquires some forcings $\mbf{f}^{\nu}$, their scale is much smaller than \eqref{eq:rapp}, as expressed by \eqref{eq:forcings}. In particular, the forcings \eqref{eq:rapp} do not even vanish as $\nu\to 0^+$ at the time where the instability reaches order one. As a consequence, the result of this chapter is indeed stronger than \autoref{main}, despite the presence of forcing terms.\\

To achieve the improved instability result, instead of truncating the series and then bounding the difference with an exact solution using energy estimates, which causes a loss of instability order, we prove the convergence of the series by utilizing \emph{generator functions}. Generator functions, which were introduced in \cite{greniernguyen}, act as uniform space-time norms encompassing derivatives of all orders. In order to preserve the required boundary conditions, the terms $\mbf{u}^n$ are obtained by progressively solving \emph{Orr-Sommerfeld equations}. These arise from
the linearized Navier-Stokes equations around $\mbf{U}_0$:
\begin{equation}\label{eq:nslinear}\begin{cases}
\partial_t \mbf{u}+(\mbf{u}\cdot \nabla)\mbf{U}_0 + (\mbf{U}_0\cdot \nabla ) \mbf{u}+\nabla p = \sqrt{\nu}\Delta \mbf{u} & \text{on }\mathbb{R}^2_+;\\ 
\nabla \cdot \mbf{u}=0 &\text{on } \mathbb{R}^2_+;\\ 
\partial_y u =\nu^{1/2-\gamma} u & y=0;\\ 
v = 0 & y=0;\end{cases}
\end{equation}
when we assume that the solution is of the form
\begin{equation}\label{eq:red}(u,v)(t,x,y)=\nabla^{\perp}\l \phi(y)e^{i\alpha(x-ct)} \r,\end{equation}
where $\alpha\in \mathbb{Z}$ represents the (spatial) frequency, and $c\in \mathbb{C}$ represents the time dependence (which will be a growth if $\Im c >0$). 
In particular the first term, $\mbf{u}^1$, is obtained from a non-trivial solution of the Orr-Sommerfeld equation without forcing,
\begin{equation}\label{eq:orrs0}\begin{cases}(U_0-c)(\partial_{yy}-\alpha^2)\phi-U_0''\phi-\frac{\sqrt{\nu}}{i\alpha}(\partial_{yy}-\alpha^2)^2\phi=0; \\
\phi(0)=0;\\
\phi'(0)=\nu^{\gamma-1/2}\phi''(0).
\end{cases}\end{equation}

The existence of such a solution for certain values of $c$ with $\Im c>0$ near the so-called Rayleigh eigenvalues, is rigorously proven in our previous paper \cite{prevpaper}. In the aforementioned paper, we also posed the question of whether an Orr-Sommerfeld eigenmode, namely a non-trivial solution to \eqref{eq:orrs0}, belongs to the image. We derived an integral criterion for this property to hold, but were unable to reach a conclusion. As a result, we are going to assume that any Orr-Sommerfeld eigenmode does not belong to the image, in line with \cite{grenier3}. 

The paper is organized as follows. In Section \ref{orrsomm} we introduce the necessary background on the Rayleigh and Orr-Sommerfeld equation with the Navier boundary condition, which was developed in \cite{prevpaper2}. We also define the pseudo-inverse for the Orr-Sommerfeld operator together with the appropriate boundary condition, extending the theory built in \cite{greniernguyen}.  In Section \ref{gen} we recall the definition of generator functions from \cite{greniernguyen} and prove that the pseudo-inverse operator is bounded with respect to them. Finally, in Section \ref{instab} we construct the unstable family of solutions $\mbf{u}^{\nu}$ to \eqref{eq:ns} starting from an unstable solution to the Orr-Sommerfeld equation and prove our main result \autoref{main}. 

\section{The Orr-Sommerfeld equation: Green function and eigenvalues}\label{orrsomm}
 In this Section, we recall a few results from our previous paper \cite{prevpaper2} on the adjoint Orr-Sommerfeld equation. Whilst most of these results were only stated for the adjoint Orr-Sommerfeld operator, the treatment of the original operator is similar, and we can apply the same techniques. Furthermore, in \cite{prevpaper2} we required the adjoint operator to deduce some results which will be applied in this paper. We remind that in this paper the scaling of the viscosity is different due to the application of the isotropic rescaling \eqref{eq:iso}. As a result, in the momentum equation the viscosity appears as $\sqrt{\nu}$ rather than $\nu$, and in the boundary condition $\nu^{\gamma}$ is transformed into $\nu^{\gamma-1/2}$ (compare \eqref{eq:ns} to \eqref{eq:ns2}).

 For each frequency $\alpha\in \mathbb{Z}$, we introduce the following operators acting on one-variable functions with sufficient regularity:
 \begin{equation}\label{eq:alphaop}
     \nabla_{\alpha}:=(\partial_y,i\alpha ),\qquad \Delta_{\alpha}:=\partial_{yy}-\alpha^2=\nabla_{\alpha}\cdot \nabla_{\alpha}.
 \end{equation}
 Let $c\in \mathbb{C}$ with $\Im c >0$, and let $U_0$ be a smooth shear flow on $\mathbb{R}_+$ satisfying the decay assumption \eqref{eq:uscond}$_1$.
Consider the Rayleigh equation \cite{rayleigh}
\begin{equation}\label{eq:ray}
\begin{cases}\ray_{c}\phi:=(U_0-c)\Delta_{\alpha}\phi-U_0''\phi=0;\\
\phi(0)=0;\end{cases}
\end{equation}
on the domain 
$$ \mc{D}=\lll \phi\in H^2(\mathbb{R}_+;\mathbb{C}):\phi(0)=0\rrr,$$
though $\ray_c\phi$ makes sense for any $\phi$ with sufficient regularity. The Rayleigh equation is of course obtained from the linearized Euler equations \eqref{eq:eulinear} by applying the ansatz \eqref{eq:red}, so that $\alpha$ represents the spatial frequency and $c$ the time dependence.

For each value of $\alpha$ and $c$, the equation \eqref{eq:ray}$_1$ admits two \emph{fundamental solutions} $\phi_c^-$ and $\phi_c^+$, with
$$\partial_y^j\phi_c^{\pm}(y)\sim (\pm |\alpha|)^j e^{\pm |\alpha|y},\qquad y\to+\infty. $$

If we linearize the Navier-Stokes equations instead of Euler we obtain the Orr-Sommerfeld equation \cite{sommerfeld} coupled with the viscosity-dependent Navier boundary condition, which is given as
\begin{equation}\label{eq:orrs}\begin{cases}\orr_{c,\nu}\phi:= (U_0-c)\Delta_{\alpha}\phi-U_0''\phi-\frac{\sqrt{\nu}}{i\alpha}\Delta_{\alpha}^2\phi=0; \\
\phi(0)=0;\\
\phi'(0)=\nu^{\gamma-1/2}\phi''(0).
\end{cases}\end{equation}
We can define the above operator on the domain
\begin{equation}\label{eq:domain} \mc{D}_{\nu}:=\lll \phi \in H^{4}(\mathbb{R}_+;\mathbb{C}):\phi'(0)=0,\;\phi'(0)=\nu^{\gamma-1/2}\phi''(0)\rrr,\end{equation}
though of course $\orr_{c,\nu}\phi$ makes sense for any function with sufficient regularity. The Orr-Sommerfeld equation is obtained from the linearized Navier-Stokes equations when we apply the ansatz \eqref{eq:red}.
 To establish the necessary properties of the image of the Orr-Sommerfeld operator, it is useful to consider its adjoint, obtained through integration by parts:
\begin{equation}
    \orr^*_{c,\nu}\phi=(U_0-\bar{c})\Delta_{\alpha}\phi+2U_0'\phi'+\frac{\sqrt{\nu}}{i\alpha}\Delta_{\alpha}^2\phi,
\end{equation}
whose domain in $L^2(\mathbb{R}_+)$ contains $\mc{D}_{\nu}$. 
\begin{prop}\label{appsol}For each $N\geq 1$ there exist four linearly independent approximate solutions of the Orr-Sommerfeld equation $\phi^{s,\pm}_N$ and $\phi^{f,\pm}_N$ such that
\begin{equation}\label{eq:appfunsol}
   \partial_y^j \phi^{s,\pm}_N(y)\sim \partial_y^j \phi_{c}^{\pm}(y)+O(\sqrt{\nu}),\qquad \partial_y^j\phi^{f,\pm}_N(y)\sim \partial_y^je^{\pm\int\mu\,\mathrm{d}y}+O(\sqrt{\nu}),\qquad \nu \to 0^+,\;j\in \mathbb{N},
\end{equation}
where
\begin{equation}\label{eq:mu}
\mu=\mu(\alpha,c,\nu,y)=|\alpha|^{1/2}\nu^{-1/4}\sqrt{\alpha \sqrt{\nu}+i(U_0(y)-c)},
\end{equation}
satisfying
$$|\orr_{c,\nu}\phi^{k,\pm}_N(y)| \leq D_N\nu^{N+1}|\phi^{k,\pm}_N(y)|,\qquad y\in \mathbb{R}_+,k=s,f.$$
\end{prop}
\begin{remark}\begin{enumerate}
\item Unfortunately, the constant $D_N$ may grow extremely quickly as $N\to \infty$, as it depends on higher order derivatives of $U_0$,
so we have been unable to obtain the convergence to exact solutions as $N\to \infty$.
\item The labels $s$ and $f$ for the approximate solutions constructed above refer to the \emph{slow} or \emph{fast} respectively behavior at infinity. Indeed when $\nu$ is small, $\mu$ will become much larger than $\alpha$.
 \end{enumerate}
 \end{remark}
 The same result of Proposition \ref{appsol} holds true for the adjoint Orr-Sommerfeld equation, except $\mu$ is replaced by its complex conjugate
 $$ \mu^*=|\alpha|^{1/2}\nu^{-1/4}\sqrt{\alpha\sqrt{\nu}-i(U_0(y)-\bar{c})}.$$
 The resulting approximate solutions are denoted as $\phi_N^{*,k,\pm}$, $k=s,f$, and they satisfy the same estimates. In fact, to first order approximation as $\nu\to 0$, we have
 $$\phi^{*,s,\pm}_N= \frac{ \bar{\phi}_N^{s,\pm}}{U_0-\bar{c}},\qquad \phi_N^{*,f,\pm}=\overline{\phi^{f,\pm}_N}.$$
  By adapting the procedure used in our previous paper \cite{prevpaper2} for the adjoint Orr-Sommerfeld operator, which is inspired from the work of Grenier and Nguyen \cite{grenier3}, we can construct an approximate interior Green function $G_{c,\nu}^I$ for the Orr-Sommerfeld equation in the form 
  $$G_{c,\nu,N}^{I}(x,y)=\begin{cases} a^{s,+}(x)\phi^{s,+}_N(y)+a^{f,+}(x)\phi^{f,+}_N(y) & y<x\\ 
  a^{s,-}(x)\phi^{s,-}_N(y)+a^{f,-}(x)\phi^{f,-}_N(y)
  & y>x.
\end{cases} $$
  As long as we do not require any boundary conditions for $G_{c,\nu}^I$, we can always  choose the coefficients $a^{s,\pm},a^{f,\pm}$
  so that
\begin{equation}\label{eq:green}\orr_{c,\nu}\int_0^{\infty}G_{c,\nu,N}^I(x,y)\psi(x)\,\mathrm{d}x=\psi(y)+O(D_N\nu^N),\end{equation}
for all $\psi \in C^3(\mathbb{R}_+)$ and such that the above integral converges for all $y\in \mathbb{R}_+$. Finally, we can correct $G_{c,\nu,N}^I$ to an \emph{exact} interior Green function $G_{c,\nu}^I$, by setting
$$G_{c,\nu}^I(x,y):=\sum_{n=0}^{\infty}\l \delta_{x=y}(x,y)\star -\orr_{c,\nu}G_{c,\nu,N}^I(x,y)\star\r^n G_{c,\nu,N}^I(x,y), $$
where the above series converges in $L^{\infty}$ with its first three derivatives in $y$ as long as $\nu$ is small enough. In particular, we obtain that $G_{c,\nu}^I$ is of class $C^2$ and with a locally bounded third derivative in $y$. Here, \emph{exact} means that $G_{c,\nu}^I$ satisfies \eqref{eq:green} without the $O(D_N\nu^N)$ term or any other remainder.

Furthermore, one can use $G_{c,\nu}^I$ to construct four \emph{exact} fundamental solutions $\phi^{s,\pm}_{c,\nu},\phi^{f,\pm}_{c,\nu}$ for the Orr-Sommerfeld equation, such that
$$ \phi^{k,\pm}_{c,\nu}=\phi^{k,\pm}_N+O(D_N\nu^N),\qquad \text{ as }\nu \to 0^+,\quad k=s,f.$$

\begin{deff} Fix $\alpha \in\mathbb{Z}\setminus \lll 0\rrr$.
    A pair $(c,\nu)$ is called an \emph{eigenvalue} for the Orr-Sommerfeld operator $\orr_{c,\nu}$ if and only if there exists $0\neq \phi_{c,\nu}\in \ker \orr_{c,\nu}\subset \mc{D}_{\nu}$. In this case, $\phi_{c,\nu}$ is called an \emph{eigenmode}. 
\end{deff}
Let us introduce the quantities
$$\mc{O}_{\gamma,c_0}(\nu):=\lim_{c\to c_0}\frac{\partial_y\phi^{s,-}_{c,\nu}(0)-\nu^{\gamma-1/2}\partial_{yy}\phi^{s,-}_{c,\nu}(0)}{\partial_y\phi^{f,-}_{c,\nu}(0)-\nu^{\gamma-1/2}\partial_{yy}\phi^{f,-}_{c,\nu}(0)},\qquad \mc{O}^*_{\gamma,c_0}(\nu):=\lim_{c\to c_0}\frac{\partial_y\phi^{*,s,-}_{c,\nu}(0)-\nu^{\gamma-1/2}\partial_{yy}\phi^{*,s,-}_{c,\nu}(0)}{\partial_y\phi^{*,f,-}_{c,\nu}(0)-\nu^{\gamma-1/2}\partial_{yy}\phi^{*,f,-}_{c,\nu}(0)}.$$
In \cite{prevpaper2} we have proven that these quantities are well-defined, even if $c_0$ is a Rayleigh eigenvalue, and continuous in $c$ and $\nu\geq 0$.
As $\nu\to 0^+$, they behave as follows. From \eqref{eq:appfunsol}, as well as Lemma 2.10 from \cite{prevpaper2} and the definition of $\mu$ we know that \begin{align*}\partial_y \phi^{s,-}_{c,\nu}(0)&\sim C_1\alpha,\quad \partial_{yy}\phi^{s,-}_{c,\nu}(0)\sim C_2(1+\alpha^2),\\
\partial_y \phi^{f,-}_{c,\nu}(0)&\sim -|\alpha|^{1/2}\nu^{-1/4}\sqrt{-ic}, \quad  \partial_{yy}\phi^{f,-}_{c,\nu}(0)\sim |\alpha|\nu^{-1/2}\sqrt{-ic},
\end{align*}
for some constants $C_1,C_2\in \mathbb{C}$, with $C_1 \neq 0$.
As a result,
\begin{equation}\label{eq:ogamma}\mc{O}_{\gamma,c_0}(\nu)\sim |\alpha|^{1/2}\nu^{1/4}\frac{C_1-C_2\alpha\nu^{\gamma-1/2}}{\sqrt{-ic}(-1-|\alpha|\nu^{\gamma-3/4})}\; \text{ as } \nu\to 0^+.\end{equation}
 The same estimate holds of course for $\mc{O}_{\gamma,c_0}^*(\nu)$, except with different values of the constants $C_1,C_2$ and $\sqrt{-ic}$ being replaced by $\sqrt{i\bar{c}}.$
 
 Now recall the following result, proven by the authors in \cite{prevpaper2}.  This will be a key ingredient in constructing the instability from the main result.
\begin{prop}\label{eigconv}A pair $(c,\nu)$ is an eigenvalue for the Orr-Sommerfeld operator if and only if it is an eigenvalue for the adjoint Orr-Sommerfeld operator. Moreover, if $c_0$ is a Rayleigh eigenvalue then there exists a $\kappa>0$ such that for $\nu>0$ small enough there exist $\kappa$ Orr-Sommerfeld eigenvalues $(c_{\nu}^1,\nu),\dots,(c_{\nu}^{\kappa},\nu)$ (counted with their multiplicity)  with
\begin{equation}|c_{\nu}^j-c_0|^{\kappa} \sim C|\mc{O}_{\gamma,c_0}(\nu)|,\qquad j=1,\dots, \kappa,\qquad \nu \to 0^+,
\end{equation}
for some $C>0$. In particular, if $\phi_{c_{\nu},\nu}$ denotes one such eigenmode, then it is of the form
\begin{equation}\label{eq:eigenexp}\phi_{c_{\nu},\nu}\sim \phi_{c_{\nu},\nu}^{s,-}+C\mc{O}_{\gamma,c_0}(\nu)\phi^{*,f,-}_{c_{\nu},\nu}\sim \phi_{c_{\nu}}^-+O(\mc{O}_{\gamma,c_0}(\nu))\sim \phi_{c_0}^-+O(\mc{O}_{\gamma,c_0}(\nu)),\end{equation}
whereas for the adjoint eigenmode $\phi^*_{c,\nu}$ we have
\begin{equation}\label{eq:eigenexpadj}\phi^*_{c_{\nu},\nu}\sim \phi_{c_{\nu},\nu}^{*,s,-}+C\mc{O}_{\gamma,c_0}(\nu)\phi^{*,f,-}_{c_{\nu},\nu}\sim \frac{\bar{\phi}_{c_{\nu}}^-}{U_0-\bar{c}_{\nu}}+O(\mc{O}_{\gamma,c_0}(\nu))\sim \frac{\bar{\phi}_{c_0}^-}{U_0-\bar{c}_{0}}+O(\mc{O}_{\gamma,c_0}(\nu)),\end{equation}
as $\nu\to 0^+$.
\end{prop}
Let $(c,\nu)$ be an Orr-Sommerfeld eigenvalue, and let $\phi_{c,\nu}$ and $\phi^*_{c,\nu}$ be an Orr-Sommerfeld eigenmode and adjoint Orr-Sommerfeld eigenmode respectively. We know that the subspace of $L^2(\mathbb{R_+})$ generated by $\phi^*_{c,\nu}$ is $\lpl \phi^*_{c,\nu}\rer= \ker \orr^*_{c,\nu}\subset \Im \orr_{c,\nu}^{\perp}$. Therefore $\Im \orr_{c,\nu}$ has at least codimension one in $L^2(\mathbb{R}_+)$, and hence the Orr-Sommerfeld operator is not surjective. Moreover, we also have
$$\phi^*_{c,\nu}\notin (\ker \orr^*_{c,\nu})^{\perp}\implies \phi^*_{c,\nu}\notin \Im \orr_{c,\nu}.$$
In other words, an adjoint Orr-Sommerfeld eigenmode cannot belong to the image of the Orr-Sommerfeld operator. Unfortunately, this is not enough to conclude the same result for an Orr-Sommerfeld eigenmode.
 \subsection{Pseudo-inverse of the Orr-Sommerfeld operator}\label{pseudo}

Let us fix $\alpha \in \mathbb{Z}\setminus \lll 0\rrr$ for the remainder of this section.  If $(c,\nu)$ is not an Orr-Sommerfeld eigenvalue, namely if the determinant of the matrix
 $$E(c,\nu):=\begin{pmatrix} \phi^{s,-}_{c,\nu}(0) & \phi^{f,-}_{c,\nu}(0) \\
 (\partial_y-\nu^{1/2-\gamma}\partial_{yy})\phi^{s,-}_{c,\nu}(0)& (\partial_y - \nu^{1/2-\gamma})\phi^{f,-}_{c,\nu}(0)
 \end{pmatrix}$$
 is non-zero, then by adding a linear combination of $\phi^{s,-}_{c,\nu}$ and $\phi^{f,-}_{c,\nu}$ to the exact interior Green function $G_{c,\nu}^I$ and obtain an exact Green function $G_{c,\nu}$ which satisfies the required boundary conditions \eqref{eq:orrs}$_{2,3}$ as well as \eqref{eq:green}. It also satisfies the following bounds:

 \begin{equation}
     \frac{|\partial_y^j G_{c,\nu}(x,y)|}{1+\det E(c,\nu)^{-1}}\leq \frac{C}{1+|\Im c|}\l |\alpha|^{j-1}e^{-\theta_0|\alpha(x-y)|}+(\mu)^{j-1}e^{-\theta_0\left|\int_x^y \mu\right|}\r,\qquad j=0,1,2,3.
 \end{equation}

Of course, the above estimate only holds for $j\leq 3$ as $\tilde{G}_{c,\nu}$ is only of class $C^2$, with a piecewise continuous third derivative. In particular, for all $p\in [1,\infty]$ there exists a constant $C(p)>0$ independent from $\nu$ such that
\begin{equation}\label{eq:l1bound0}\int_0^{\infty} |\partial_y^jG_{c,\nu}(x,y)|^p\,\mathrm{d}x\leq C(p)\l 1+E(c,\nu)^{-1}\r\l |\alpha|^{j-2}+\mu^{j-2}\r,\qquad j=0,1,2,3. \end{equation}
 While it is not possible to obtain a Green function satisfying the boundary condition for each pair $(c,\nu)$, namely for those for which $\det E(c,\nu)=0$, we can always fix one of the two conditions and allow the other condition to restrict the image of the operator to a subspace of codimension one. Regardless of the choice, one obtains that $\Im \orr_{c,\nu}$ is the same subspace of codimension at most one. On the other hand, since $\Im \orr_{c,\nu}\subset (\ker \orr^*_{c,\nu})^{\perp}$, and since eigenvalues and adjoint eigenvalues are the same, one deduces that the codimension of $\Im \orr_{c,\nu}$ is exactly one.

Recalling that $\phi_{c,\nu}^{f,-}(0)\neq 0$ for all $(c,\nu)$, we can then define
\begin{equation}
    \tilde{G}^b_{c,\nu}(x,y):= -\frac{G_{c,\nu}^I(x,0)}{\phi^{f,-}_{c,\nu}(0)}\phi^{f,-}_{c,\nu}(y),
\end{equation}
so that the \emph{pseudo-Green function} Let $\tilde{G}_{c,\nu}:=G^I_{c,\nu}+\tilde{G}^b_{c,\nu}$ is a Green function for the Orr-Sommerfeld operator, satisfying the boundary condition $\tilde{G}_{c,\nu}(x,0)=0$ for all $x>0$. In our previous paper \cite{prevpaper2}, we have proven the following properties.
\begin{prop}The pseudo-Green function $\tilde{G}_{c,\nu}$ satisfies the same estimates as the Green function $G_{c,\nu}^I$: for all $\theta_0\in(0,1)$ there exists $C>0$ independent from $\nu$, such that for all $x,y>0$ we have
\begin{equation}\label{eq:pseudobound}
|\partial_y^j\tilde{G}_{c,\nu}(x,y)|\leq \frac{C}{1+|\Im c|}\l |\alpha|^{j-1}e^{-\theta_0|\alpha(x-y)|}+(\mu)^{j-1}e^{-\theta_0\left|\int_x^y \mu\right|}\r, \qquad j=0,1,2,3.
\end{equation}
In particular, for all $p\in [1,\infty]$ there exists a constant $C(p)>0$ independent from $\nu$ such that
\begin{equation}\label{eq:l1bound}\int_0^{\infty} |\partial_y^j\tilde{G}_{c,\nu}(x,y)|^p\,\mathrm{d}x\leq C(p)\l |\alpha|^{j-2}+\mu^{j-2}\r,\qquad j=0,1,2,3. \end{equation}
\end{prop}

In line with \cite{grenier3}, we are going to require the assumption that $c_0$ is a \emph{simple} eigenvalue of Rayleigh, in the sense that if $\phi_{c_0}^-$ is the fundamental decaying solution of $\ray_{c_0}$, then 
 \begin{equation*}
     \phi_{c_0}^-\notin \Im \ray_{c_0}.
 \end{equation*}
 This is equivalent to requiring that
$$ \ker \ray^2_{c_0}=\ker \ray_{c_0}.$$
As argued in our paper \cite{prevpaper2}, this also implies that for $\nu$ small enough we have
\begin{equation}
    \phi_{c_{\nu},\nu}\notin \Im \orr_{c_{\nu},\nu},
\end{equation}
where $\phi_{c_{\nu},\nu}$ is the eigenmode corresponding to an Orr-Sommerfeld eigenvalue $(c_{\nu},\nu)$ with $c_{\nu}\to c_0$ as $\nu\to 0$. Under the above condition, we can conclude the following.
\begin{prop}\label{orrim}Let $(c,\nu)$ be an Orr-Sommerfeld eigenvalue, let $\tilde{G}_{c,\nu}$ be the pseudo-Green function, and let $\psi \in C^{3}\cap L^2(\mathbb{R}_+)$. Then $\psi \in \Im \orr_{c,\nu}$, with preimage $$\phi(y):=\int_0^{\infty}\tilde{G}_{c,\nu}(x,y)\psi(x)\,\mathrm{d}x.$$
if and only if:
\begin{equation}\label{eq:orthogorr}\int_0^{\infty}(\partial_y-\nu^{\gamma-1/2}\partial_{yy})\tilde{G}_{c,\nu}(x,0)\psi(x)\,\mathrm{d}x=0.
\end{equation}

In particular, $\Im \orr_{c,\nu}$ has codimension one. 
\end{prop}

\textbf{Notation}. We will use the notation $A(\nu) \propto B(\nu)$, where $A(\nu)$ to $B(\nu)$ are some complex quantities to mean that $A(\nu) \sim C B(\nu)$ as $\nu\to 0^+$, where $C=C(\alpha,c,\nu)\in \mathbb{C}$ is continuous and bounded away from zero in $\alpha \in \mathbb{Z}$, $c\in \mathbb{C}$ and $\nu >0$.

\begin{remark}
  Let  $(c,\nu)$ be an Orr-Sommerfeld eigenvalue. Then we know that $E(c,\nu)=0$, so the eigenmode is of the form
   \begin{equation}\label{eq:eigenexp2} \phi_{c,\nu}\propto \phi_{c,\nu}^{s,-}+ \frac{(\partial_y- \nu^{1/2-\gamma})\phi_{c,\nu}^{s,-}(0)}{(\partial_y-\nu^{1/2-\gamma})\phi_{c,\nu}^{f,-}(0)}\phi_{c,\nu}^{f,-}\sim \phi_c^-+O\l \mc{O}_{\gamma,c}(\nu)\r,\qquad \nu \to 0^+,\end{equation}
   where $\phi_c^-$ is the Rayleigh decaying solution.
    The above expression may be compared to \eqref{eq:eigenexp}, where we take the limit as $c$ approaches a nearby Rayleigh eigenvalue $c_0$, which leads to the  $\mc{O}_{\gamma,c_0}(\nu)$ term. However, \eqref{eq:eigenexp2} holds regardless of the presence of such a Rayleigh eigenvalue.\\

    Similarly, the adjoint eigenmode is of the form
    \begin{equation}\label{eq:eigenexpadj2} \phi^*_{c,\nu}\propto \phi_{c,\nu}^{*,s,-}+ \frac{(\partial_y- \nu^{1/2-\gamma})\phi_{c,\nu}^{*,s,-}(0)}{(\partial_y-\nu^{1/2-\gamma})\phi_{c,\nu}^{*,f,-}(0)}\phi_{c,\nu}^{*,f,-}\sim \frac{\overline{\phi_c^-}}{U_0-\bar{c}}+O\l \mc{O}^*_{\gamma,c}(\nu)\r,\qquad \nu\to 0^+.
\end{equation}
\end{remark}
Let $(c,\nu)$ be an Orr-Sommerfeld eigenvalue, and let $\phi_{c,\nu}$ be an eigenmode. By our assumption that $\phi_{c,\nu}\notin \Im \orr_{c,\nu}$, we can then define a projector to $\Im \orr_{c,\nu}$ as 
\begin{align}
\label{eq:proj} P_{c,\nu}:L^2(\mathbb{R}_+)&\to \Im \orr_{c,\nu}\\
\nonumber \phi &\mapsto \phi- \lambda_{c,\nu}\phi_{c,\nu},
\end{align}
where $\lambda_{c,\nu}$ is the unique real number such that $\phi-\lambda_{c,\nu}\phi_{c,\nu}\in \Im \orr_{c,\nu}$. In fact, we have
$$\lambda_{c,\nu}=\frac{\int_0^{\infty}(\partial_{y}-\nu^{\gamma-1/2}\partial_{yy})\tilde{G}_{c,\nu}(x,0)\phi(x)\,\mathrm{d}x}{\int_0^{\infty}(\partial_{y}-\nu^{\gamma-1/2}\partial_{yy})\tilde{G}_{c,\nu}(x,0)\phi_{c,\nu}(x)\,\mathrm{d}x}.$$
Proposition \ref{orrim} ensures that the denominator in the above expression cannot vanish for $\nu$ small enough.
A simpler formula can be obtained by recalling that
$$\Im \orr_{c,\nu}\subset {\ker \orr^*_{c,\nu}} ^{\perp}= \lpl \phi^*_{c,\nu} \rer^{\perp},$$
and so
$$0= \lpl \phi- \lambda_{c,\nu}\phi_{c,\nu},\phi^*_{c,\nu}\rer_{L^2} \iff \lambda_{c,\nu}= \frac{\lpl \phi, \phi^*_{c,\nu}\rer_{L^2}}{\lpl \phi_{c,\nu},\phi^*_{c,\nu}\rer_{L^2}}.$$
We can finally define the \emph{pseudo-inverse} of the Orr-Sommerfeld operator.
\begin{deff} Let $(c,\nu)$ be an Orr-Sommerfeld eigenvalue. We define the \emph{pseudo-inverse} of the Orr-Sommerfeld operator $\orr_{c,\nu}$ as
\begin{align*}
    \tilde{\orr}_{c,\nu}^{-1}:C^{3}_b(\mathbb{R}_+)&\to \mc{D}_{\nu}\\
    \psi &\mapsto \int_0^{\infty}\tilde{G}_{c,\nu}(x,y)P_{c,\nu}\psi(x)\,\mathrm{d}x,
\end{align*}
so that $\tilde{Orr}^{-1}\psi\in \mc{D}_{\nu}$, and
\begin{equation}
\orr_{c,\nu}\tilde{\orr}_{c,\nu}^{-1}\psi= P_{c,\nu}\psi=\phi-\lambda_{c,\nu}\phi_{c,\nu}.
\end{equation}
\end{deff}

\section{Generator functions}\label{gen}

To construct the instability, we will carry out our analysis using the so-called generator functions. We will use the same generator functions which were employed to construct the instability for the no-slip case in \cite{grenier3}. We summarize the construction here, focusing on some aspects which are of particular relevance to this paper.

For a function $g:\mathbb{R}_+\to \mathbb{R}$, $\delta \geq 0$ and $\ell \in \mathbb{N}_{\geq 0}$, define the norm
$$\|g\|_{\ell,\delta}:=\begin{cases} \sup_{y\geq 0}\varphi(y)^{\ell}|g(y)|\l \delta^{-1}e^{-y/\delta}+1\r^{-1}& \delta>0,\\
\sup_{y\geq 0}\varphi(y)^\ell |g(y)|& \delta=0;
\end{cases}\qquad \varphi(y):=\frac{y}{1+y}.$$
Now define the weight function
$$w_{\ell,\delta}(y):=\varphi(y)^{\ell}\l\delta^{-1}e^{-y/\delta}+1\r^{-1},\qquad y\geq 0.  $$
Then $w_{\ell,\delta}$ is increasing in $y\geq 0$ and $w_{\ell,\delta}(0)=0$ for all $\ell>0$.

For $\ell=0$, we have $w_{\ell,\delta}=\l\delta^{-1}e^{-y/\delta}+1\r^{-1}\in \left[ \frac{1}{1+\delta^{-1}},1\right].$ Therefore, the $\|\cdot\|_{0,\delta}$ norms control the $L^{\infty}$ norm. However, the $\|\cdot\|_{\ell,\delta}$ norms do not control the $L^{\infty}$ norms of any derivatives, since for $\ell \geq 1$, $ \varphi(y)$ appears with a positive exponent and $\varphi(0)=0$.  Another property is that for any $k\in \mathbb{N}$ and for any function $g$,
\begin{equation}
    \|g\|_{\ell+k,\delta}=\|\varphi^k w_{\ell,\delta}g\|_{L^{\infty}}\leq \|w_{\ell,\delta}g\|_{L^{\infty}}=\|g\|_{\ell,\delta}.
\end{equation}

As for the role of $\delta$, we have the following relation between the $\|\cdot\|_{\ell,\delta}$ and $\|\cdot\|_{\ell,0}$ norms:
\begin{equation}\label{eq:gnorms}
    \frac{1}{1+\delta^{-1}}\|\cdot\|_{\ell,0}\leq \|\cdot \|_{\ell,\delta}\leq \|\cdot\|_{\ell,0}.
\end{equation}
In other words, the $\|\cdot\|_{\ell,\delta}$ and $\|\cdot\|_{\ell,0}$ norms are equivalent for any fixed $\ell$. However, in our application we will interpret $\delta$ as 
\begin{equation}\label{eq:delta}\delta=\gamma_0\nu^{1/4},\end{equation}
for some $\gamma_0>0$. We will simplify the notation by writing \eqref{eq:delta} as $\delta \propto \nu^{1/4}$ from now on. A consequence of \eqref{eq:delta} is that the mapping between the  $\|\cdot\|_{\ell,\delta}$ and $\|\cdot\|_{\ell,0}$ norms is not uniformly bounded as $\nu\to 0$.
\begin{deff}\label{gendef}
Let $g=g(y):\mathbb{R}_+\to \mathbb{R}$ be a smooth function. We define the \emph{one-variable generator function} as
\begin{equation}\label{eq:gendelta}\gen_{\delta}g(z_2):=\sum_{\ell \geq 0}\|\partial_y^{\ell}g\|_{\ell,\delta}\frac{z_2^{\ell}}{\ell!}.\end{equation}
Let now $f=f(x,y):\mathbb{T}\times \mathbb{R}_+\to \mathbb{R}$ be a two-variable function, smooth in $x$ and analytic in $y>0$. For $\alpha \in \mathbb{Z}$, let $f_{\alpha}(y)$ be the Fourier transform in $x$ of $f$. For $\delta,z_1,z_2\geq 0$ we can define the \emph{two-variable generator function} as
\begin{equation}\label{eq:gendelta2}\gen_{\delta}f(z_1,z_2):=\sum_{\alpha \in \mathbb{Z}}\sum_{\ell \geq 0}e^{z_1|\alpha|}\|\partial_y^{\ell}f_{\alpha}\|_{\ell,\delta}\frac{z_2^\ell}{\ell!},\qquad \forall \delta \geq 0.\end{equation}
\end{deff}
\begin{remark} Following the notation of Definition \ref{gendef}:
\begin{enumerate}
\item From \eqref{eq:gendelta} and \eqref{eq:gendelta2} it follows that 
$$\gen_{\delta} f(z_1,z_2)=\sum_{\alpha \in \mathbb{Z}}e^{z_1|\alpha|}\gen_{\delta}\l z_2\mapsto f(z_1,z_2)\r.$$

\item The maps $z_1\mapsto \gen_{\delta}f(z_1,z_2)$, $z_2\mapsto \gen_{\delta} f(z_1,z_2)$ and $z_2\mapsto \gen_{\delta} g(z_2)$ are all increasing.

\item From \eqref{eq:gnorms} we deduce that in the two-variable case we have $\gen_{0}f(z_1,z_2)<\infty$ if and only if $\gen_{\delta}f(z_1,z_2)<\infty$ for any (and all) $\delta>0$, and similarly in the one-variable case.
\item 
If $\gen_{\delta}f<\infty$ for some $\delta \geq 0$, then $f$ is analytic in $y$ except possibly for $y=0$. In addition, there is a limit on how the norms of derivatives of $f$ can grow. For the above series to converge, one needs a constant $C>0$ such that
$$\|\partial_y^{\ell}f\|_{\ell,\delta}\leq C^{\ell}\ell!\;.$$
\end{enumerate}
\end{remark}

We refer to \cite{greniernguyen}, Section 3 for a more detailed discussion on the general properties of the generator functions. 

\begin{remark}\label{genscale}
An important characteristic of generator functions is their behavior with respect to \emph{boundary layer}-type functions, which results from the $\varphi^{\ell}$ factor in the $\|\cdot\|_{\ell,\delta}$ norms. Let $\delta>0$, and let $f$ be smooth and rapidly decaying. Then \begin{equation}\label{eq:deltaresc}\|\partial_y^{\ell}\l f(y/\delta)\r\|_{\ell,\delta}\leq \delta^{-\ell}\sup_{y\geq 0}y^{\ell}|(\partial_y^{\ell}f)(y/\delta)|=\sup_{y\geq 0}y^{\ell} |f(y)|.\end{equation}
This shows that the generator function $\gen_{\delta}$ of $y\mapsto f(y/\delta)$  can be bounded by a constant independent from $\delta$. More generally, by \eqref{eq:gnorms}, this extends to functions of the type $y\mapsto g(y)+f(y/\delta)$, where $g$ and $f$ are smooth and rapidly decaying. We call these \emph{boundary layer} functions.
\end{remark}

In the following Proposition, we establish that the projection operator defined in \eqref{eq:proj} is continuous with respect to the generator functions. For this, we introduce the constraint $|\alpha|\leq \nu^{-1/4}$, which will be assumed to be satisfied for any pair $(\alpha,\nu)$ for the remainder of the paper, unless otherwise specified. 
\begin{prop}\label{projbound} Let $\delta\propto \nu^{1/4}$. 
   There exists a constant $C>0$, independent from $\alpha,\nu$ and continuous in $c$ as long as $\Im c \gg 0$, such that for all $\phi \in  L^2(\mathbb{R}_+)$ with $\gen_0\phi <\infty$ we have
   \begin{equation}\label{eq:projbound}
    \gen_\delta P_{c,\nu}\phi \leq C \gen_{\delta}\phi.
\end{equation}
\end{prop}
\begin{proof}Let $\phi_{c,\nu}$ and $\phi_{c,\nu}^*$ be the eigenmode and adjoint eigenmode, normalized so that they satisfy \eqref{eq:eigenexp2} and \eqref{eq:eigenexpadj2}. Let us also assume that the Rayleigh solution $\phi_c^-$ is normalized in the $L^2$ norm.  We have
$$\|\partial_y^{\ell}(\phi-\lambda_{c,\nu}\phi_{c,\nu})\|_{\ell,\delta}\leq \|\partial_y^{\ell}\phi\|_{\ell,\delta}+|\lambda_{c,\nu}|\|\partial_y^{\ell}\phi_{c,\nu}\|_{\ell,\delta}.$$
The first term in the right-hand side is immediately bounded by $\gen_{\delta}\phi$. As for the second term, first notice that since $\delta \propto\nu^{1/4}$ then $\|\partial_y^{\ell}\phi_{c,\nu}\|_{\ell,\delta}$ is bounded by a constant independent from $\nu$ (see \eqref{eq:deltaresc}). To bound $\lambda_{c,\nu}$, recall that we assume $|\alpha| \leq \nu^{-1/4}$ and notice that as $\nu\to 0^+$, by \eqref{eq:eigenexp2} and \eqref{eq:eigenexpadj2} we have
$$|\lpl \phi_{c,\nu},\phi^*_{c,\nu}\rer| =\left| \lpl \phi^{-}_{c},\frac{\overline{\phi^{-}_{c}}}{U_0-\bar{c}}\rer\right|+O(\nu^{1/4}\alpha^{1/2})\gtrsim \frac{1}{\sup|U_0-\bar{c}|}\geq \frac{1}{\|U_0\|_{L^{\infty}}+|c|}.$$
Also, by the above normalization assumptions as well as the first part of \eqref{eq:eigenexpadj2}, as $|\alpha|\to \infty$ and $\nu\to 0^+$ the $L^1$ norm of $\phi^*_{c,\nu}$ satisfies
$$\|\phi^*_{c,\nu}\|_{L^1}\lesssim \sup\frac{1}{|U_0-\bar{c}|} |\alpha|^{-1/2}+|\alpha|^{1/4}\nu^{1/8}.  $$
Therefore
\begin{align*}
    |\lambda_{c,\nu}|&=\left|\frac{\lpl \phi, \phi^*_{c,\nu}\rer_{L^2}}{\lpl \phi_{c,\nu},\phi^*_{c,\nu}\rer_{L^2}}\right|\lesssim (\|U_0\|_{L^{\infty}}+|c|)\|\phi^{*}_{c,\nu}\|_{L^1}\|\phi\|_{L^{\infty}}\leq (\|U_0\|_{L^{\infty}}+|c|)\frac{C_1|\alpha|^{-1/2}+C_2\nu^{1/8}|\alpha|^{1/4}}{\Im c}\|\phi\|_{L^{\infty}}\\
    &\leq \frac{\|U_0\|_{L^{\infty}}+|c|}{\Im c}\gen_{\delta}\phi.
\end{align*}
for some $C>0$ independent from $\alpha,c$ and $\nu$, as long as $\nu$ is small enough. This yields \eqref{eq:projbound},.
\end{proof}

To conclude this section, we show that there exists at least one linearly unstable profile satisfying the conditions given in \eqref{eq:uscond}.

\begin{example}
Consider the profile
$$U_0(y)=e^{-1/y}. $$
We have already proven in \cite{prevpaper} that $U_0$ is linearly unstable for the Euler equations. It also clearly satisfies \eqref{eq:uscond}$_1$ and \eqref{eq:uscond}$_2$. We show that it also
satisfies \eqref{eq:uscond}$_3$, with constant $C=2$. The $k$-th derivative of $U_0$ is given by
$$\dv{}{y^k} e^{-1/y}= k!e^{-1/y}\sum_{j=0}^{k-1}\frac{(-1)^j}{(k-j)!}\begin{pmatrix} k-1\\ j\end{pmatrix}y^{-2k+j}. $$
 We then have
\begin{align*}
    \sup_{y>0}\left|\frac{y^k}{(1+y)^k}\dv{}{y^k}e^{-1/y}\right|&\leq k!\sum_{j=0}^{k-1} \frac{1}{(j+1)!}\begin{pmatrix}k-1\\ j\end{pmatrix} \sup_{y>0}e^{-1/y}y^{-1-j}\\
    &\leq k! \sum_{j=0}^{k-1}\frac{1}{(j+1)!}\begin{pmatrix} k-1\\ j\end{pmatrix} \l \frac{j+1}{e}\r^{j+1}.
\end{align*}

But asymptotically as $k\to \infty$, the last term becomes
$$\sim k!\sum_{j=0}^{k-1}\frac{1}{\sqrt{2\pi (j+1)}} \begin{pmatrix}k-1\\ j\end{pmatrix}\leq k!\sum_{j=0}^{k-1}\begin{pmatrix} k-1\\ j\end{pmatrix} = 2^{k-1}k!.$$
This concludes the proof.
\end{example}
The above example illustrates how condition 
\eqref{eq:uscond}$_3$  is strictly weaker than analyticity, even when all the other assumptions are taken into account.

Recall that $\delta\propto \nu^{1/4}$, where $\nu>0$ is the viscosity. We hereby provide what in our view is a simplified proof compared to \cite{greniernguyen}.

\begin{prop}\label{normbound}Let $(c,\nu)$ not be an Orr-Sommerfeld eigenvalue. Let $\phi\in \mc{D}_{\nu}$ be a solution to the Orr-Sommerfeld equation \eqref{eq:orrs} with a forcing $R$ such that $\|R\|_{0,0}<\infty$.
Then there exists a constant $C>0$ such that for $\alpha \leq \nu^{-1/4}$ and $|\Im c| \gg 0$, we have
\begin{align}
    (1+|\Im c|)\l \|\nabla_{\alpha}\phi\|_{0,0}+\|\Delta_{\alpha}\phi\|_{0,\delta}\r+\frac{\sqrt{\nu}}{|\alpha|}\|\Delta_{\alpha}^2\phi\|_{0,\delta}\leq C \|R\|_{0,\delta}.
\end{align}
\end{prop}
\begin{proof} We just need to prove
$$  (1+|\Im c|)\l \|\nabla_{\alpha}\phi\|_{0,0}+\|\Delta_{\alpha}\phi\|_{0,\delta}\r\leq C \|R\|_{0,\delta},$$
as the estimate on the remaining term follows from \eqref{eq:orrs}. Following \cite{greniernguyen}, let us split $R$ as 
$$ R= R^b+R^I,$$
where
$$R^b(y)=R(y)\chi_{[0,1]}\l \frac{y}{\delta\log\delta^{-1}}\r.$$
In this case, we have $\|R^I\|_{0,0}\leq 2\|R\|_{0,\delta}.$ We may now set
$$\phi^b(y):=\int_0^{\infty}G_{c,\nu}(x,y)R^b(x)\,\mathrm{d}x, \qquad \phi^I(y):=\int_0^{\infty}G_{c,\nu}(x,y)R^I(x)\,\mathrm{d}x.$$
For $\phi^I$, by \eqref{eq:l1bound0} we have
$$\|\phi^I\|_{0,0} \leq \|R\|_{0,0}\int_0^{\infty}|G_{c,\nu}(x,y)|\,\mathrm{d}x\leq C\|R\|_{0,0}\l |\alpha|^{-2}+\mu^{-2}\r.$$
As for $\phi^b$, notice that $R^b(x)=0$ for $x>\delta\log\delta^{-1}$. Therefore, we obtain
\begin{align*}
\|\phi^b\|_{0,0}&\leq \|R\|_{0,\delta}\int_0^{\delta\log\delta^{-1}}|G_{c,\nu}(x,y)|(1+\delta^{-1}e^{-x/\delta})\,\mathrm{d}x\leq \|R\|_{0,\delta}\|G_{c,\nu}\|_{L^{\infty}}(1-\delta +\delta\log\delta^{-1})\\
&\leq \frac{C}{1+|\Im c|}(\alpha^{-1}+\mu^{-1})\|R\|_{0,\delta},
\end{align*}
where $C$ does not depend on $\delta$ as long as $\delta$ (i.e. $\nu$) is small enough. Finally, by \eqref{eq:gnorms} we conclude the estimate for the $\|\cdot\|_{0,\delta}$ norms.

\end{proof}
The above estimate can then be extended to derivatives of arbitrary orders and thus to generator functions.

\begin{corollary}\label{genbound}Let $\phi_{\alpha}\in \mc{D}_{\nu}$ be a solution to the Orr-Sommerfeld equation \eqref{eq:orrs} with a forcing $R_{\alpha}$ such that $\|R_{\alpha}\|_{0,0}<\infty$. Then there exists a constant $C>0$ such that for $|\alpha|\leq \nu^{-1/4}$ and $|\Im c|\gg 0$, we have
\begin{equation*}
   \gen_0 \l \nabla_{\alpha}\phi_{\alpha}\r+ \gen_{\delta}\l \Delta_{\alpha}\phi_{\alpha} \r\leq \frac{C}{1+|\Im c|}\gen_{\delta}R_{\alpha};
\end{equation*}
moreover, if
\begin{gather*} \phi(t,x,y):=\sum_{|\alpha|\leq \nu^{-1/4}}\phi_{\alpha}(y)e^{i\alpha(x-\Re ct)},\qquad (u,v):=\nabla^{\perp}\phi,\qquad \omega:=\Delta \phi,\\
 R(t,x,y):=\sum_{|\alpha|\leq \nu^{-1/4}}R_{\alpha}(y)e^{i\alpha(x-\Re ct)},\end{gather*}
then taking the generator functions with respect to the $x$ and $y$ variables, we have
\begin{align*}
\partial_{z_1}^j\l \gen_0 u + \gen_0 v+ \gen_{\delta}\omega \r& \leq \frac{C}{1+|\Im c|}\partial_{z_1}^j\gen_{\delta} R,\qquad j\in \lll 0;1\rrr,\\
\partial_{z_2}\l \gen_0 u +\gen_0 v +\gen_{\delta}\omega \r&\leq \frac{C}{1+|\Im c|}\l \partial_{z_2}\gen_{\delta} R+\gen_{\delta}R\r,
\end{align*}
for $z_2$ small enough, independently from $\nu$ and $\alpha$.
\end{corollary}
\begin{proof} The proof is the same as in \cite{grenier2}. Note that in our case the shear flow $U_0$ is not necessarily analytic, but the proof only relies on the assumption that the series
$$ \sum_{n=0}^{\infty}\l \|\partial_y^nU_0\|_{L^{\infty}}+\|\partial_y^{n+2}U_0\|_{L^{\infty}}\r\frac{z^n}{n!}$$
is convergent for $z$ small enough. This follows from our assumption \eqref{eq:uscond}$_4$ instead.
\end{proof}
The results obtained in this section thus imply that we can rely on the same estimates as the paper \cite{grenier2} for the solutions to \eqref{eq:orrs}.

\subsection{The Orr-Sommerfeld equation for $\alpha=0$}\label{alpha00}
When $\alpha=0$, the Orr-Sommerfeld equation is not defined according to \eqref{eq:orrs0}. We wish to consider a particular extension arising from the assumption that $c=\Re c_1+\frac{in}{\alpha}\Im c_1$ when $\alpha\neq 0$, where $n\in \mathbb{N}_{\geq 1}$, $\Im c_1>0$ (the meaning of these parameters will be made clear later). To derive it, we simply multiply the Orr-Sommerfeld equation \eqref{eq:orrs} by $\alpha$ and then set $\alpha=0$. Introducing a forcing $R$, which we assume to be integrable and with $\gen_0 R<\infty$, we reach the equations
\begin{equation}\label{eq:alpha0}
   \begin{cases} n\Im c_1 \phi'' -\sqrt{\nu}\phi^{(4)}=R & \text{ on }\mathbb{R}_+;\\
   \phi(0)=0;\\
   \phi'(0)=\nu^{\gamma-1/2}\phi''(0).
   \end{cases}
\end{equation}

Our goal is to extend the estimates of Corollary \ref{genbound} to the above equation.  Since $\alpha=0$, we only need to control the generator functions of $\phi'$ and $\phi''$. In fact, for our application we only need the estimates on $\phi''$.
\begin{prop}\label{alpha0} Let $\phi$ be a solution to \eqref{eq:alpha0}. Let $u(y)=\phi'(y)$, and $\omega(y)=u'(y)$. Then 
\begin{align*}
    \partial_{z_1}^j\gen_{\delta}\omega&\leq \frac{C}{1+n\Im c_1}\partial_{z_1}^j\gen_{\delta}R,\qquad j\in \lll 0;1\rrr,\\
    \partial_{z_2}\gen_{\delta}\omega& \leq \frac{C}{1+n\Im c_1}\l \partial_{z_2}Gen_{\delta}R + \gen_{\delta} R\r,
\end{align*}
for $z_2$ small enough, independently from $\nu$, $n$ and $\alpha$.\end{prop}
\begin{proof}

We know that \eqref{eq:alpha0} is equivalent to
$$\begin{cases} n\Im c_1 \omega(y)- \sqrt{\nu}\omega''(y) =R(y) & y \geq 0;\\ 
u'(y)=\omega(y)& y\geq 0;\\
u(0)=\nu^{\gamma-1/2}\omega(0).
\end{cases} $$
A decaying solution to the first equation is given by
$$\omega(y)=\int_0^{\infty}G_0(x,y)R(x)\,\mathrm{d}x,\qquad \beta=\sqrt{n\Im c_1}\nu^{-1/4},$$
where the Green function $G_0$ is
$$G_0(x,y)=-\frac{1}{2\beta \sqrt{\nu}}e^{-\beta|y-x|}.$$

In this case we have
$$\int_0^{\infty}|G_0(x,y)|\,\mathrm{d}x=\frac{1}{2\beta^2\sqrt{\nu}}\l 2-e^{-\beta y}\r\leq \frac{1}{n\Im c_1},$$
so the coefficients of $G_0$ are bounded uniformly in $\nu$. 
In other words, they satisfy the same bounds as in the Dirichlet case, which can be applied as required in the construction of Section \ref{instab} (see Proposition 5.5 from \cite{grenier3}). This yields the bounds for $\omega$. 
\end{proof}
Obtaining bounds for $U_0$ is not obvious due to the lack of a term of order lower than $2$ in \eqref{eq:alpha0}. It requires the additional assumption that $R$ is in boundary layer form, as defined at the end of Remark \ref{genscale}. In this case, the Laplace equation estimates obtained in Section 4.2 of \cite{grenier3} allow us to estimate the generator functions of $\phi$ and $U_0$ through the generator functions of $\omega$.

\section{Construction of the instability}\label{instab}

Let $N_0>0,s\geq 0,\gamma>1/2$, and let $U_0\in C^{\infty}(\mathbb{R}_+;\mathbb{R})$ be a linearly unstable shear flow for the linearized Euler equation, i.e. such that the associated Rayleigh equation \eqref{eq:ray} admits at least an eigenvalue $c_0$ with $\Im c_0>0$ for some $\alpha_0\in\mathbb{Z}$, and satisfying the assumptions \eqref{eq:uscond}. Consider then the family $\lll U^{\nu}\rrr_{\nu>0}$ given as the classical solutions to the heat equation \eqref{eq:shflows}. Recall that, by Lemma 4.2 from \cite{prevpaper}, we know that for $\gamma >1/2$ we have for some constant $C>0$, for $t>0$ small enough  and for all $s\geq 0$,
$$\|U^{\nu}(\sqrt{\nu}t)-U_0\|_{H^s} \leq C\nu^{\gamma-1/2}. $$ However, we do not know whether $\left.U^{\nu}\right|_{t=0}$ is linearly unstable for Orr-Sommerfeld. Therefore, we shall construct the leading order term of our solution starting from $U_0$.

Let $\alpha_0\in\mathbb{Z}$ be such that $c_0\in \mathbb{C}$ is a maximally unstable eigenvalue for the associated Rayleigh equation, namely such that $\alpha_0\Im c_0$ is maximum across all Rayleigh eigenvalues for the shear flow $U_0$. The existence of such an $\alpha_0$ is granted by Theorem 4.1 in \cite{grenier}.
By Proposition \ref{eigconv}, we know that for all $\nu>0$ small enough there exists $c_1=c_0+O(\nu^{1/4})$ such that $(c_1,\nu)$ is an eigenvalue for the Orr-Sommerfeld equation. Then, let $\psi^1$ be a corresponding eigenmode. We can use $\psi^1$ to define a real valued solution of the linearized Navier-Stokes equations around $U_0$:
\begin{equation}\label{eq:u1}\mbf{u}^1(t,x,y):= \nabla^{\perp}\l e^{i\alpha_0 (x-\Re c_1t)}\psi^1(y) \r+ \nabla^{\perp}\l e^{-i\alpha_0 (x-\Re c_1t)} \overline{\psi^1(y)}\r. \end{equation}
In particular, $\mbf{u}^1$ is given by the sum of two components with two spatial frequencies, $\alpha_0$ and $-\alpha_0$ respectively.\\

We require that $\psi^1$ be normalized so that
\begin{equation}\label{eq:normaliz}\|\mbf{u}^1(t)\|_{L^{\infty}}=\|\Re \nabla^{\perp}\psi^1\|_{L^{\infty}}=1\qquad \forall \nu>0,t\geq 0.\end{equation}
By the results of Section \ref{orrsomm}, $\psi^1$ may be expanded as follows as $\nu\to 0^+$:
$$\psi^{1}(y)\sim\frac{1}{C}\l\phi_{c_0}^-(y)+O(\mc{O}_{\gamma,c_0}(\nu))e^{-\int_0^y \Re \mu}\r$$
where $\phi_{c_0}^-$ is the Rayleigh eigenmode corresponding to the eigenvalue $c_0$, and $C\neq 0$ is independent from $\nu$. As a consequence, the (isotropically rescaled) $H^s$ norms scale as
$$\|\psi^1\|_{H^s}\lesssim   \nu^{\min \lll 0;3/8-s/4\rrr }.$$ 

Now, if we return to the original variables (which were used in the statement of  \autoref{main}) by reversing the isotropic rescaling:
$$(t,x,y)\mapsto \sqrt{\nu}(t,x,y).$$
In the original variables, due to the rescaling of the $H^s(\mathbb{T}\times \mathbb{R}_+)$ norms, we then have
$$\|\psi^1\|_{H^s}\lesssim \nu^{\min \lll 0; 7/8-3s/4\rrr}.$$

\uline{From now on, for simplicity we shall assume that $\alpha_0=1$}. We wish to contruct a solution of the form
\begin{equation}\label{eq:ansatz}
    \mbf{w}^{\nu}(t,x,y):=\sum_{n=1}^{\infty}\tau^n \mbf{u}^n(t,x,y), \quad \mbf{u}^n(t,x,y):=\nabla^{\perp} \psi^n(y),\quad \psi^n(t,x,y):=\sum_{|\alpha|\leq n\land \nu^{-1/4}}e^{i\alpha(x-\Re c_1 t)}\psi_{\alpha}^n(y),
\end{equation}
where
\begin{equation}\label{eq:tau}\tau=\nu^{N_s} e^{\Im c_1 t},\qquad N_s :=N_0+\max \lll 0;3s/4-7/8\rrr,\end{equation}
such that $\mbf{u}^{\nu}(t,x,y):=\mbf{w}^{\nu}(t,x,y)+\mbf{U}^{\nu}(\sqrt{\nu}t,y)$ solves the Navier-Stokes equations with forcing \eqref{eq:ns2}.
\begin{remark}\begin{enumerate}
\item Thanks to our choice of $N_s$, at time $t=0$ we have as $\nu\to 0^+$, in the original (non isotropically rescaled) variables:
$$\|\left.(\mbf{u}^{\nu}-\mbf{U}_0)\right|_{t=0}\|_{H^s}\sim  \nu^{N_s}\|\left.\mbf{u}^1\right|_{t=0}\|_{H^s}\leq \nu^{N_0},$$
so that the assumption \eqref{eq:main0} from \autoref{main} is automatically satisfied, assuming the sums in \eqref{eq:ansatz} converge in $H^s$. 
\item Since $\Im c_1>0$ in \eqref{eq:u1}, then $\mbf{u}^1$ will grow exponentially in time and this will produce the instability in \autoref{main}.
\item In the ansatz \eqref{eq:ansatz} we have $\psi^n\in \mathbb{R}$, and hence also $\mbf{u}^n=(u^n,v^n)$ and $\mbf{w}^{\nu}=(w_1,w_2)$ are real valued. We choose to truncate the expansion for $\psi^n$ at $|\alpha|\leq \nu^{-1/4}$ because Proposition \ref{normbound} only holds within this range. The sum will of course include the null frequency $\alpha=0$.
\end{enumerate}
\end{remark}

Similarly to our paper \cite{prevpaper}, let $\Lambda\in \mathbb{N}$, $\Lambda \geq 1$ be such that
$$ 2^{-\Lambda} \leq \gamma-\frac{1}{2}.$$
Then $\nu^{-2^{-\Lambda}}(U^{\nu}(0,y)-U_0(y))$ is bounded as $\nu\to 0$, and $\mbf{w}=(w_1,w_2)=\mbf{w}^{\nu}$ must satisfy the following equations:

\begin{equation}\label{eq:idk2}\begin{cases}\partial_t \mbf{w} + (\mbf{U}_0\cdot \nabla) \mbf{w} + ( \mbf{w}\cdot \nabla) \mbf{U}_0+( \mbf{w}\cdot \nabla)  \mbf{w}= \sqrt{\nu}\Delta  \mbf{w}+ \nu^{2^{-\Lambda}}S \mbf{w}+\mbf{f};\\ 
\nabla \cdot  \mbf{w}=0;\\
\partial_y  \mbf{w}\cdot \tau = \nu^{1/2-\gamma} \mbf{w}\cdot \tau;\\ 
 \mbf{w}\cdot n =0;
\end{cases} \end{equation}
where
\begin{equation}\label{eq:sdef} S \mbf{w}(t):= \frac{\mbf{U}_0(y)- \mbf{U}^{\nu}(\sqrt{\nu}t,y)}{\nu^{2^{-\Lambda}}}\cdot \nabla  \mbf{w}(t) +  \mbf{w}(t)\cdot \nabla \l \frac{\mbf{U}_0(y)- \mbf{U}^{\nu}(\sqrt{\nu}t,y)}{\nu^{2^{-\Lambda}}} \r, \end{equation}
and $\mbf{f}$ is yet to be determined.
This formulation is convenient because the left hand side of \eqref{eq:idk2} is now a linearized Navier-Stokes equation around a \emph{time-independent} shear flow, which can be reduced to the well-known Rayleigh and Orr-Sommerfeld equations.

Let us initially assume that the evolution of the shear flow is time independent, so that $S\mbf{w}=0$. In terms of the vorticity ${\omega}$, by taking the scalar curl of \eqref{eq:idk2}$_1$, we obtain  
\begin{equation}\label{eq:idk3}\begin{cases}
\partial_t {\omega} + U_0\partial_x{\omega} -U_0'' w_2 + (\mbf{w}\cdot \nabla){\omega} =\sqrt{\nu}\Delta {\omega}+g;\\
{\omega}= \nu^{1/2-\gamma}\mbf{w}\cdot \tau;\\
\mbf{w}\cdot n=0;
\end{cases}
\end{equation}
where $g=\partial_y f_1-\partial_x f_2$. Let us introduce the new quantities:
$$\tilde{\psi}^n:= e^{n\Im c_1 t}\psi^n;\qquad \tilde{\mbf{u}}^n:=e^{n\Im c_1 t}\mbf{u}^n;\qquad \tilde{\omega}^n:=e^{n\Im c_1 t}\omega^n;\qquad  \omega^n:=\Delta \psi^n. $$
In this case, the ansatz \eqref{eq:ansatz} may be rewritten as
\begin{equation}
\mbf{w}^{\nu}(t,x,y)=\sum_{n=1}^{\infty}\nu^{nN_s} \tilde{\mbf{u}}^n.
\end{equation}

The equation satisfied by $\tilde{\omega}^n$ is obtained by equating the terms with order in $\nu$ belonging to the interval $N_s [n,n+1)$ in \eqref{eq:idk3}. We obtain
\begin{equation}\label{eq:idkn}
    \partial_t \tilde{\omega}^n +U_0\partial_x\tilde{\omega}^n-U_0''\partial_x\tilde{\psi}^n-\sqrt{\nu}\Delta\tilde{\omega}^n=-\sum_{1\leq j\leq n}(\tilde{u}^j\cdot \nabla)\tilde{\omega}^{n-j}+\tilde{g}^n,
\end{equation}
where the $\tilde{g}^n$ are such that $g$ satisfies the expansion $g=\sum_{n=1}^{\infty}\nu^{nN_s}\tilde{g}^n.$

Let us now derive the equation satisfied by ${\psi}_{\alpha}^n$ for $n\geq 2$, which we have introduced in \eqref{eq:ansatz}, and determine the appropriate forcing $g$. Define
$$\mbf{u}_{\alpha}^n(y)=(u_{\alpha}^n(y),v_{\alpha}^n(y)):=(\partial_y\psi_{\alpha}^n,-i\alpha \psi_{\alpha}^
n), \quad \omega_{\alpha}^n(y):=\Delta_{\alpha}\psi_{\alpha}^n=\partial_{yy}\psi_{\alpha}^n-\alpha^2\psi_{\alpha}^n $$
so that 
\begin{gather}\label{eq:psin}\psi^n=\sum_{|\alpha|\leq n\land \nu^{-1/4}}\psi^n_{\alpha}(y)e^{i\alpha (x-\Re c_1t)},\quad \mbf{u}^n = \sum_{|\alpha|\leq n\land \nu^{-1/4}}\mbf{u}^n_{\alpha}(y)e^{i\alpha (x-\Re c_1t)},\nonumber \\
\omega^n=\sum_{|\alpha|\leq n\land \nu^{-1/4}}\omega_{\alpha}^n(y) e^{i\alpha(x-\Re c_1 t)}. \end{gather}
Of course, this implies that $\mbf{u}^n=\nabla^{\perp}\psi^n$ and $\omega^n= \Delta \psi^n. $ 
By plugging these expressions into \eqref{eq:idkn} (neglecting the forcing $\tilde{g}^n$ for now) and equating the terms which share the same frequency $\alpha$, we obtain the equations
\begin{equation}\label{eq:orreqn}
    \begin{cases}\displaystyle
    \orr_{c_{n,\alpha},\nu}\psi_{\alpha}^n=- \frac{1}{i\alpha}\sum_{|\alpha'|\land \leq n\land \nu^{-1/4}}\sum_{1\leq j\leq n-1}\l \mbf{u}_{\alpha'}^j\cdot \nabla_{\alpha}\r \omega_{\alpha-\alpha'}^{n-j} & \alpha \neq 0;\\
        \displaystyle n\Im c_1 (\psi_0^n)'' -\sqrt{\nu}(\psi_{0}^n)^{(4)}=-\sum_{|\alpha|\leq n\land \nu^{-1/4}}\sum_{1\leq j\leq n-1}\l \mbf{u}_{\alpha}^j \cdot \nabla_{\alpha}\r \omega_{-\alpha}^{n-j}.
\end{cases}
\end{equation}

where
\begin{equation}\label{eq:cdef} c_{n,\alpha}= \Re c_1 +i n\frac{\alpha_0}{\alpha}Im c_1,\qquad \alpha\neq 0.
\end{equation}

Therefore, we can set our forcing terms $g^n$ to be
\begin{equation}\label{eq:forcing} g^n:=   \sum_{1\leq j\leq n-1}\sum_{\nu^{-1/4}<|\alpha_1+\alpha_2|\leq n\land  2\nu^{-1/4}}e^{i\alpha(x-\Re c_0 t)}\l \mbf{u}_{\alpha_1}^j\cdot \nabla_{\alpha_2}\r \omega^{n-j}_{\alpha_2},\end{equation}
so that \eqref{eq:idkn} will be satisfied. 

\begin{remark}\begin{enumerate}
\item If $n=1$, then $|\alpha|\leq n$ implies $|\alpha|\leq 1$. In fact, given our definition \eqref{eq:u1}, $\mbf{u}^1$ only consists of two modes with respective frequency $\alpha=\pm \alpha_0=\pm 1$. These satisfy $\mbf{u}^1_1=\overline{\mbf{u}^1_{-1}}$, and therefore $\psi^1_{1}=\overline{\psi^1_{-1}}$. 
Unfortunately, this property does not hold for higher values of $n$ because the Orr-Sommerfeld equation has complex coefficients.

\item Since $|\alpha|\leq n$, we have $g^n=0$ unless $n> \nu^{-1/4}$. Moreover, $g^n$ is only comprised of modes with $|\alpha|\leq 2\nu^{-1/4}$. 

\item If $\Im c_1 >0$, which is the instability condition, then $\Im c_{n,\alpha} = \frac{n}{\alpha}\Im c_1 >0$ for all $n>0$. Thus $\Im c_{n,\alpha}$ is bounded away from $0$, and therefore $|U_0-c_{n,\alpha}|$ is also bounded away from $0$ for all the $c_{n,\alpha}$ considered. As a result, we remain away from any so-called critical layer. Secondly, it is important to point out that $\psi_\alpha^n=0$ unless $|\alpha|\leq n|\alpha_0|=n$. This can be checked by induction on \eqref{eq:orreqn} given that $\psi_\alpha^1=0$ unless $\alpha =\pm \alpha_0.$
\end{enumerate}
\end{remark}
As for $\alpha =0$, the associated modes will satisfy a simple second-order equation (see section \ref{alpha00}).
Notice that the operator $\orr_{c_{n,\alpha},\nu}$ is invertible for all $(\alpha,n)\neq (\alpha_0,1)=(1,1)$, so the equation can be solved together with the boundary conditions. In addition, we know that the solution $\psi_{\alpha}^n$ belongs to $H^s(\mathbb{R}_+)$ for all $s\geq 0$. 

We then obtain that $\omega:=\sum_{n=1}^{\infty}\omega^n \tau^n$ solves \eqref{eq:idk3} with the forcing
\begin{equation}\label{eq:f}g:=\sum_{n=1}^{\infty}g^n \tau^n.\end{equation}

 \begin{lemma}\label{genlemma}
     Let $h$ be a smooth function on $\mathbb{R}\times \mathbb{R}_+$, let $\mbf{u}=(u,v)$ be a smooth divergence free function on $\mathbb{R}\times \mathbb{R}_+$, and let $\mbf{u}^n,\omega^n$ be defined as in \eqref{eq:psin}. Define
\begin{align*}
\mc{A}&:=(\text{Id} + \partial_{z_1}+\partial_{z_2})\gen_{\delta};\\
\mc{B}(\mbf{u},h)&:=\gen_0 u + \gen_0 v+\partial_{z_1}\gen_0u+\mc{A}h;\\
G^n&:=\mc{B}(\mbf{u}^n,\omega^n). \end{align*}
Then we have, for $z_2$ small enough (see \cite{greniernguyen}, Proposition 3.3) and for all $j,k\in \mathbb{N}$,
\begin{align}\label{eq:aineq2} G^n &\leq C\mc{A}(\omega^n)\\
\label{eq:aineq}\mc{A}\l \sum_{j=1}^{n-1}\mbf{u}^j\cdot \nabla \omega^{n-j}\r&\leq C\sum_{j=1}^nG^j (\partial_{z_1}+\partial_{z_2})G^{n-j}, \end{align}
for some $C>0$ independent from $n$.
 \end{lemma}
\begin{proof} To prove \eqref{eq:aineq2} we just need to show 
\begin{equation*}
    \gen_0 u^n + \gen_0 v^n + \partial_{z_1}\gen_0 u^n \leq C \l  \gen_{\delta}\omega^n + \partial_{z_1}\gen_{\delta}\omega^n +\partial_{z_2}\gen_{\delta}\omega^n\r.
\end{equation*}
This follows from the generator function estimates on the Laplace operator established in \cite{greniernguyen}. Note that the aforementioned estimates only rely on the Dirichlet boundary condition for the stream function, which is still satisfied in our framework.\\

To prove \eqref{eq:aineq}, recall that by Proposition 3.2 from \cite{greniernguyen}, we have, for any divergence-free flow field $(u,v)$ and scalar function $h$, and $|z_2|\leq 1$:
\begin{align}\label{eq:prop1}
    \gen_{\delta}(u\partial_x h)&\leq \gen_0 u\,\partial_{z_1}\gen_{\delta}h;\\
    \label{eq:prop2}
    \gen_{\delta}(v\partial_y h)&\leq C\l \gen_0 v+\partial_{z_1}u\r \partial_{z_2}\gen_{\delta}h,\qquad C>0.
\end{align}
By splitting $\mc{A}$ in its three components $$\mc{A}=\gen_{\delta}+\partial_{z_1}\gen_{\delta}+\partial_{z_2}\gen_{\delta},$$ we need to bound each of these terms by the right-hand side of \eqref{eq:aineq} for some positive constant $C>0$.  For the computations below we will mainly rely on \eqref{eq:prop1} and \eqref{eq:prop2}, but for the full details we refer the reader to the proof of Proposition 3.3 from \cite{greniernguyen}.
\begin{enumerate}
\item  First component: \begin{align*}\sum_{1\leq j\leq n-1}\gen_{\delta}(\mbf{u}^j\cdot \nabla \omega^{n-j})&\leq \sum_{1\leq j\leq n-1}\left[\gen_0u^j\partial_{z_1}+ C \l \gen_0 v^j+\partial_{z_1}\gen_0u^j\r\partial_{z_2}\right]\gen_{\delta}\omega^{n-j} \\
&\leq \sum_{1\leq j\leq n-1}\mc{B}(\mbf{u}^j,\omega^j)(\partial_{z_1}+\partial_{z_2})\mc{B}(\mbf{u}^{n-j},\omega^{n-j})\\
&=\sum_{1\leq j\leq n-1}G^j(\partial_{z_1}+\partial_{z_2})G^{n-j}.
\end{align*}
\item Second component: 
    \begin{align*}\sum_{1\leq j\leq n-1}\partial_{z_1}\gen_{\delta}(\mbf{u}^j\cdot \nabla \omega^{n-j})&\leq \sum_{1\leq j\leq n-1}\left[\partial_{z_1}\gen_0 u^j\partial_{z_1}\gen_{\delta}\omega^{n-j}+\gen_0 u^j\partial_{z_1}^2\gen_{\delta}\omega^{n-j}\right]\\
    & + C\sum_{1\leq j\leq n-1}\l \partial_{z_1}\gen_0v^j+\partial_{z_1}^2\gen_0u^j\r \partial_{z_2}\gen_{\delta}\omega^{n-j}\\ 
    &+ C\sum_{1\leq j\leq n-1}\l \gen_0v^j+\partial_{z_1}\gen_0u^j\r \partial^2_{z_2z_1}\gen_{\delta}\omega^{n-j}\\
    &\leq \sum_{1\leq j\leq n-1}G^j\partial_{z_1}G^{n-j}+C\sum_{1\leq j\leq n-1}\partial_{z_1}G^j G^{n-j}+C\sum_{1\leq j\leq n-1}G^j\partial_{z_2}G^{n-j}\\
   & \leq (C+1)\sum_{1 \leq j\leq n-1}G^j(\partial_{z_1}+\partial_{z_2})G^{n-j}.
\end{align*}
\item Third component: 
\begin{align*}
\sum_{1\leq j\leq n-1}&\partial_{z_2}\gen_{\delta}\l \mbf{u}^j\cdot \nabla \omega^{n-j}\r\leq \sum_{1\leq j\leq n-1}\l \partial_{z_2}\gen_0 u^j \partial_{z_1}\gen_{\delta}\omega^{n-j}+C_0 \gen_0u^j\partial^2_{z_1z_2}\gen_{\delta}\omega^{n-j}\r\\
&+ \sum_{1\leq j\leq n-1}\l \partial_{z_1}\gen_0 u^j\partial_{z_2}\gen_{\delta}\omega^{n-j}+C_0 \l \gen_0 v^j+\partial_{z_1}\gen_0 u^j\r \partial_{z_2}^2\gen_{\delta}\omega^{n-j}\r\\
& \leq \sum_{1\leq j\leq n-1}\l \partial_{z_2}G^{j}G^{n-j}+ C_0 G^j \partial_{z_1}G^{n-j}+ G^j \partial_{z_2}G^{n-j}+ C_0 G^j \partial_{z_2}G^{n-j} \r\\
&\leq 2C_0 \sum_{1\leq j\leq n-1}G^j (\partial_{z_1}+\partial_{z_2})G^{n-j}.
\end{align*}
\end{enumerate}
Therefore, each of the three terms defining $\mc{A}$ satisfies \eqref{eq:aineq}, and hence their sum also does, after adapting the constant $C$ in the right hand side.
\end{proof}
 
\begin{prop} Let $\omega^n$ and $g^n$, $n\in \mathbb{N}_{\geq 1}$ be defined as above. Then the series
$$\sum_{n=1}^{\infty}\omega^n\tau^n,\qquad \sum_{n=1}^{\infty}g^n\tau^n$$
converge in the sense of generator functions if $\tau=\nu^{N_s}e^{\Im c_0 t}$ is smaller than some positive constant independent from $\nu$. 
\end{prop}
\begin{proof}We just need to connect the estimates \eqref{eq:aineq2} and \eqref{eq:aineq}. To do this, we apply Corollary \ref{genbound} (for $\alpha\neq 0$) and Proposition \ref{alpha0} (for $\alpha=0$) to \eqref{eq:orreqn}, recalling that $\nabla_{\alpha}\psi_{\alpha}^n=((\psi_{\alpha}^n)',i\psi_{\alpha}^n)=\mbf{u}_{\alpha}^n$ and $\Delta_{\alpha}\psi_{\alpha}^n=\omega_{\alpha}^n$ by definition.  For $j=1,\dots, n$, let $G^j:=\mc{B}(\mbf{u}^j,\omega^j)$ following the notation from Lemma \ref{genlemma}.  Then for $n\geq 2$, by Lemma \ref{genlemma} we have
$$G^n \leq C\mc{A}(\omega^n)\leq \frac{C_0}{n}\mc{A}\l \sum_{1\leq j\leq n-1}\mbf{u}^j\cdot \nabla \omega^{n-j}\r\leq \frac{C}{n}\sum_{1\leq j\leq n-1}\l G^j \l  \partial_{z_1}+\partial_{z_2}\r G^{n-j}\r, $$
from which it follows that the series
$$\sum_{n=1}^{\infty}G^n(z_1,z_2)\tau^{n-1} $$
converges for $\tau,z_1,z_2$ small enough.\\ 

Now recall that
\begin{align*}
    \gen_{\delta}(u^j\partial_x \omega^k)&\leq \gen_0(u^j)\partial_{z_1}\gen_{\delta}(\omega^k);\\ 
    \gen_{\delta}(v^j\partial_y \omega^k)&\leq C\l \gen_0(v^j)+\partial_{z_1}\gen_0(u^j)\r\partial_{z_2}\gen_{\delta}(\omega^k).
\end{align*}
As a result, for some $C>0$ we have
$$\gen_{\delta}(u^j \cdot \nabla \omega^k)\leq (C+1)G^jG^k. $$
Therefore,
$$\gen_{\delta}\l \sum_{n=1}^{\infty}g^n\tau^n \r \leq \tau^{\nu^{-1/4}}\gen_{\delta}\l \sum_{n=1}^{\infty}g^n \tau^{n-\nu^{-1/4}}\r\leq \tau^{\nu^{-1/4}}(C+1)C^2\sum_{n>\nu^{-1/4}}^{\infty}\tau^{n-\nu^{-1/4}}\leq C\tau^{\nu^{-1/4}}. $$
We conclude that the series $\sum_{n=1}^{\infty}\tau^ng^n $ converges in the space of functions with $\gen_{\delta}<\infty$ for $z_1,z_2$ small enough.
\end{proof}

 \subsection{Time dependent shear flow}\label{timedep}
 We now complete the construction started in the previous subsection by considering the time and viscosity dependence of the shear flows. This generates some higher order terms in $\nu$ and $t$. 
 
By Proposition 2.1 and Proposition 2.2 from \cite{prevpaper}, $S\mbf{w}$ is smooth and bounded, uniformly in $\nu$ and $t\geq 0$ in all spatial $H^s$ norms. For this, the assumption that $U_0^{(k)}(0)=0$ for all $k\in\mathbb{N}$ is necessary.  In the no-slip case, it suffices to assume that $U_0^{(2k)}(0)=0$ for all $k\in \mathbb{N}$; in both cases, this corresponds to enforcing the compatibility conditions for all orders (see the Appendix).

Recall that the shear flows $U^{\nu}$ satisfy the heat equation \eqref{eq:shflows}. For each integer $M>0$ we can then write, for $\gamma >1/2,$ the following Taylor expansion near $\sqrt{\nu}t=0$:
 \begin{equation}\label{eq:flowexp}
     U^{\nu}(\sqrt{\nu}t,y)=U_0(y)+  \sum_{k=1}^M(\nu^{1/2}t)^kU_0^{(2k)}(y) + R_U^{\nu}(t,y),\qquad R_U^{\nu}(t,y)=O((\nu^{1/2}t)^{M+1}).
 \end{equation}
 Recall that $U^{\nu}(\sqrt{\nu}t,y)$, and hence $R_U^{\nu}(t,y)$, is bounded in any $H^s$ norm, uniformly as $\nu\to 0^+$. This follows from Proposition 2.1 and 2.2 in \cite{prevpaper}. We will keep the value of $M>0$ from \eqref{eq:flowexp} fixed until the end of the Section.
 
 Consider the vorticity formulation of the Navier-Stokes equations:
 \begin{equation}\label{eq:nsvort}\begin{cases}
     \partial_t \omega +(\mbf{U}_0\cdot \nabla)\omega+(\mbf{u}\cdot \nabla)U_0'+(\mbf{u}\cdot \nabla)\omega-\sqrt{\nu}\Delta \omega =  \nu^{2^{-\Lambda}}S\omega +g ;\\ 
     v=0 & y=0; \\
     \omega = \nu^{1/2-\gamma}\partial_{y}\omega & y=0;
     \end{cases}
 \end{equation}
 where
 $$ S\omega=-\frac{U^{\nu}(\sqrt{\nu}t,y)-U_0(y)}{\nu^{2^{-\Lambda}}}\partial_x\omega -(\mbf{u}\cdot \nabla)\frac{\partial_y U^{\nu}(\sqrt{\nu}t,y)-U_0'(y)}{\nu^{2^{-\Lambda}}}.$$
We are looking for a solution of the form
\begin{equation}\label{eq:sol} \psi(t,x,y)= \sum_{n\geq 1}\sum_{p=0}^{2M}\sum_{q=0}^M\sum_{|\alpha| \leq n\land \nu^{-1/4}}\tau^n t^p \nu^{q2^{-\Lambda}}e^{i\alpha(x-\Re c_1t)}\psi_{\alpha}^{n,p,q}(y),\end{equation}
 where $\mbf{u}^{\nu}:=\nabla^{\perp}\psi$ is such that $\mbf{u}^{\nu}(t,x,y)+\mbf{U}^{\nu}(\sqrt{\nu}t,y) $ solves the Navier-Stokes equations with the Navier boundary condition \eqref{eq:ns2}, or equivalently $\omega:=\Delta\psi$ satisfies \eqref{eq:nsvort}. This will be the family of solutions displaying the instability of our main result \autoref{main}.
 
 Substituting \eqref{eq:sol} into the Navier-Stokes equations and equating terms with the same order 
 $t^p\nu^{q2^{-\Lambda}}e^{n\Im c_0 t}$, we obtain the equations
 \begin{equation}\label{eq:recursfinal}
     i\alpha \orr_{c_{n,\alpha},\nu}(\psi_{\alpha}^{n,p,q})+(p+1)\psi_{\alpha}^{n,p+1,q}=Q^{n,p,q}+L^{n,p,q},
 \end{equation}
 for $n\geq 1 ,0\leq q\leq M,0 \leq p\leq 2M$ where
 \begin{align*}
     Q^{n,p,q}&=\sum_{|\alpha'|\leq n\land \nu^{-1/4}}\sum_{1\leq j\leq n-1}\sum_{0\leq k\leq p}\sum_{0\leq l\leq q}(\mbf{u}_{\alpha-\alpha'}^{j,k,l}\cdot \nabla_{\alpha'})\omega_{\alpha'}^{n-j,p-k,q-l};\\
     L^{n,p,q}&= i\alpha \sum_{k=1}^M U_0^{(2k)}\omega_{\alpha}^{n,p-k,q-2^{\Lambda-1}k}-i\alpha \sum_{k=1}^M \psi_{\alpha}^{n,p-k,q-2^{\Lambda-1}k}U_0^{(2k+2)},
 \end{align*}
 and $c$ is defined as in \eqref{eq:cdef}. However, the terms with powers strictly higher than $M$ in $t$ or $\nu^{2^{-\Lambda}}$ (including the remainder from \eqref{eq:flowexp}) will not appear in any of the equations \eqref{eq:recursfinal}, and will instead end up in a remainder:
\begin{align*}R^n_{\alpha}(t,y):&=\sum_{|\alpha'|\leq n\land \nu^{1/4}}\sum_{1\leq j\leq n-1} \sum_{p_1,p_2< 2M}(\star)+ i\alpha R^{\nu}_U(t,y)\partial_x\omega_{\alpha}-i\alpha \psi_{\alpha}\partial_{yy} R^{\nu}_U(t,y),\\
 (\star)&=\begin{cases} \displaystyle \sum_{q_1+q_2> M}^{q_1,q_2\leq M}{\nu}^{(q_1+q_2)2^{-\Lambda}}t^{p_1+p_2}(\mbf{u}_{\alpha-\alpha'}^{j,p_1,q_1}\cdot \nabla_{\alpha'})\omega_{\alpha'}^{n-j,p_2,q_2}& |\alpha|\leq n\land \nu^{-1/4};\\
  \displaystyle \sum_{q_1,q_2\leq M}{\nu}^{(q_1+q_2)2^{-\Lambda}}t^{p_1+p_2}(\mbf{u}_{\alpha-\alpha'}^{j,p_1,q_1}\cdot \nabla_{\alpha'})\omega_{\alpha'}^{n-j,p_2,q_2} & \nu^{-1/4}<|\alpha| \leq n\land 2\nu^{-1/4}.
 \end{cases}
\end{align*}
As usual, we define
\begin{equation} \label{eq:remainder}R^n(t,x,y)= \sum_{|\alpha|\leq n\land 2\nu^{-1/4}}e^{i\alpha (x-\Re c_0 t)}R^n_{\alpha}(t,y), \qquad R=\sum_{n=1}^{\infty} \tau^n R^n, \end{equation}
so that the forcing $g$ from \eqref{eq:nsvort} is given by $g=\Delta R$.

 Notice that while we could truncate the sum for $p$ and $q$ in our ansatz \eqref{eq:sol}, we cannot do so for $n$. Indeed, while $\tau$ can be made arbitrarily small, it cannot be asymptotically smaller than $\nu^{N_s}$. On the other hand, taking $M>2^{\Lambda}N_s$,  we obtain a factor of $\nu^{N_s}$ in the terms above. As for $p$, we shall prove that if $\psi_{\alpha}^{n,p,q}\neq 0$ then $p\leq 2q$. Hence all terms with $p> 2M$ will vanish. Moreover, if $q_1+q_2\leq M$, then $p_1+p_2< 2(q_1+q_2)\leq 2M$, so those terms will not appear in the remainder.

 Throughout this discussion we assume, by simplicity, that $\alpha_0$ is the only integer such that $c_0$ is a Rayleigh eigenmode, and that $c_{n,\alpha}$ (as defined in \eqref{eq:cdef} is not an eigenmode for any $n> 1$.
 
 To find all the terms $\psi_{\alpha}^{n,p,q}$, we first claim that the following ansatz is consistent with equation \eqref{eq:recursfinal}:
 $$\psi_{\alpha}^{n,p,q}=0 \qquad \forall p> 2q.
  $$
 
 To see why this is the correct range, notice that the terms in $Q^{n,p,q}$ are products of terms with indexes $(j,k,l)$ and $(n-j,p-k,q-l)$, with $1\leq j\leq n-1$. Now supposing these terms satisfy the above claim, we can deduce that $Q^{n,p,q}=0$ for $p> 2q$. Indeed, the non-zero products must satisfy both $k\leq 2l$ and
 $$2q-k < p-k\leq 2(q-l) \iff k>2l, $$
 which is impossible. 
 
 As for $L^{n,p,q}$, it contains terms with indexes $(n,p-k,q-2^{\Lambda-1}k)$ for $k\geq 1$. Suppose these lower order terms satisfy the claim. Then for $p>2q$,  the non-zero terms will satisfy
 $$ 2q -k< p-k \leq  2(q-2^{\Lambda-1} k)\iff k<0, $$
 which is impossible. In conclusion, $Q^{n,p,q}=L^{n,p,q}=0$ in the specified range if the claim is true for lower order terms in at least one between $n$ and $q$.  This implies that we can prove the claim by induction over $(n,q)$, once we have dealt with the base cases $q=0$ and $n=1$.
 \begin{itemize}
     \item \uline{$q=0$}. In this case, $L^{n,p,0}=0$ by construction. Moreover, $Q^{n,p,0}$ consists of terms with indexes $(n',p',0)$ where $n'<n$. For $n=1$, in particular we have $Q^{1,p,0}=0$. We can then set $\psi_{\alpha_0}^{1,1,0}=0$ because $\orr_{c,\nu}\psi_{\alpha_0}^{1,0,0}=0$. Suppose now the claim is true for $n'\leq n$. Then $Q^{n,p,0}=0$ for all $p>0$. Hence the claim is consistent for all $p>0$, recalling that for $n>1$ the Orr-Sommerfeld operator is invertible.
     \item \uline{$n=1$}. In this case, $Q^{1,p,q}=0$ by construction. Moreover, $L^{1,p,q}$ consists of terms with indexes $(1,p-k,q-k)$, for $k\geq 1$. The case $q=0$ has been treated above, so suppose the claim is true for $q'<q$, $q\geq 1$. Let $p\geq 2q$. Then the non-zero terms of $L^{1,p,q}$  must satisfy $2q-k\leq p-k\leq 2(q-2^{\Lambda-1}k)$, but this is impossible for $k\geq 1$. Thus $L^{1,p,q}=0$, and the equation is consistent. For $p=2q-1\geq 1$, we get $L^{1,2q-1,q}\neq 0$ (choosing $k=1$). If $\alpha\neq \alpha_0$, since the Orr-Sommerfeld operator is invertible, the equation can then be subsequently solved for $\psi^{1,p,q}$ for all $p\leq 2q-1$. Otherwise if $\alpha=\alpha_0$, the Orr-Sommerfeld operator is non-invertible, so we define $$\psi^{1,2q-1,q}:=\tilde{\orr}_{c_{n,\alpha},\nu}^{-1}(L^{1,2q-1,q})$$ and, if $\phi\mapsto \lambda_{c,\nu}(\phi)$ denotes the quantity defined in \eqref{eq:proj}, then $$2\psi^{1,2q,q}:=\lambda_{c,\nu}(L^{1,2q-1,q})\lambda_{c,\nu}(\phi_{c_1,\nu})^{-1}\phi_{c_1,\nu}, $$ where $\phi_{c_1,\nu}$ is the Orr-Sommerfeld eigenmode. Since then $\orr_{c_{n,\alpha},\nu}\psi^{1,2q,q}=0$, this is consistent with $L^{1,2q,q}=0$. The terms for $p<2q-1$ may be derived by progressively solving Orr-Sommerfeld equations, using the pseudo-inverse. Proposition \ref{projbound} will allow us to control the size of the solutions each time the pseudo-inverse is applied.
 \end{itemize}

\begin{remark}
    The case $n=1$ above is where we have finally used the assumption that $\phi_{c,\nu}\notin \Im \orr_{c,\nu}$.
    It is also where the condition $p>2q$ for the ansatz $\psi_{\alpha}^{n,p,q}=0$ becomes necessary. If we used $p>q$ for the ansatz instead, then we would have to make an exception for the case $\alpha=\pm \alpha_0$,  since evidently e.g. $\psi_{\alpha_0}^{1,2,1}\neq 0$. However, it also follows that $\psi_{2\alpha_0}^{2,2,1}\neq 0$, even though $2\alpha_0\neq \pm \alpha_0$ and $p=2>1=q$.  The inconsistency caused by this cascade effect to higher order terms can only be avoided with the condition $p>2q$. 
\end{remark}

Let us now show that all the modes can be actually derived. We have already shown this for $q=0$ and $n=1$. We proceed by induction over $(n,q)$. Suppose we have constructed $\psi_{\alpha}^{n',p,q'}$ for all $n'\leq n,q'\leq q$ and for all $p$. We construct $\psi_{\alpha}^{n+1,p,q'}$ for $q'\leq q+1$ and $\psi_{\alpha}^{n',p,q+1}$ for $n'\leq n+1$. In the first case, we proceed by induction over $q'$. We have already constructed all the $\psi_{\alpha}^{n+1,p,0}$. Suppose we have constructed $\psi_{\alpha}^{n+1,p,q'}$ for $q'< q_0$. For $p>2q_0$ we know that $\psi_{\alpha}^{n+1,p,q_0}=0$. $Q^{n+1,p,q_0}$ and $L^{n+1,p,q_0}$ contain terms which are lower order in $n$ or $q_0$ respectively, so these all vanish by inductive assumption. Hence the equations in this range reduce to $0=0$. Now consider $p=2q_0$. Then $Q^{n,p,q}+L^{n,p,q}\neq 0$ and we can define $\psi_{\alpha}^{n,p,q}:=\orr_{c,\nu}^{-1}(Q^{n,p,q}+L^{n,p,q})$, recalling that $n\geq 2$ and hence the Orr-Sommerfeld operator is invertible. For $\psi_{\alpha}^{n',p,q+1}$ the procedure is the same. This completes the induction.
  
  \subsubsection{Bounds}

 We now derive bounds for the generator functions of the solutions. Recall that the sums for $p$ and $q$ are truncated to $2M$ and $M$ respectively, where $M$ is a fixed quantity. Therefore, the bounds need only be uniform in $n$, not $p$ or $q$. We use $C_{p,q}$ to indicate a constant that may depend on $p$ and $q$ but not on $n$.  We define
 $$\psi^{n,p,q}(t,x,y):=\sum_{|\alpha|\leq \nu^{-1/4}\land n} e^{i\alpha(x-\Re c_1t)}\psi_{\alpha}^{n,p,q}(y),\quad \mbf{u}^{n,p,q}:=\nabla^{\perp}\psi^{n,p,q};\quad \omega^{n,p,q}:=\Delta \psi^{n,p,q},$$
 \begin{align*}G^{n,p,q}&:=\mc{B}(\mbf{u}^{n,p,q},\omega^{n,p,q})\\
&=\gen_0(u^{n,p,q})+\gen_0(v^{n,p,q})+\partial_{z_1}\gen_0(u^{n,p,q})+\gen_{\delta}(\omega^{n,p,q})+\partial_{z_1}\gen_{\delta}(\omega^{n,p,q})+\partial_{z_2}\gen_{\delta}(\omega^{n,p,q}). \end{align*}
Notice that, by equation \eqref{eq:recursfinal} as well as Remark \ref{genscale} (recall that $\delta \propto\nu^{1/4}$), each $G^{n,p,q}$ is continuous in $\nu$, converging to a non-zero quantity as $\nu\to 0^+$ if $p\leq 2q$.

We wish to prove the following result.

 \begin{prop}\label{genconv}Let
\begin{equation}G^n:=\sum_{q=0}^M\sum_{p=0}^{2q} t^p\nu^{q2^{-\Lambda}}G^{n,p,q} .
\end{equation}
     Then the series of positive real numbers
     \begin{equation}\label{eq:mainseries}
         \sum_{n=1}^{\infty}\tau^{n-1}G^n\qquad \text{ and }\qquad \sum_{n=1}^{\infty}\tau^{n-1}R^n
     \end{equation}
     converge for all $\tau <C_{\infty}$, for some $C_{\infty}>0$ independent from $\nu$. Furthermore, the value of each sum is continuous in $\nu$.
 \end{prop}
 Define
 $$
G_{N}:=\sum_{n=0}^N\tau^{n-1}G^n=\sum_{n=1}^N\sum_{q=0}^M\sum_{p=0}^{2q} \tau^{n-1}t^p\nu^{q2^{-\Lambda}}G^{n,p,q} . $$
Suppose we can prove that there exists a constant $C_M$, continuous in $\nu$ and such that 
\begin{equation}\partial_{\tau}G_N\leq C_MG_N\l \partial_{z_1}+\partial_{z_2}\r G_N.
\end{equation}
From this, proceeding as in \cite{greniernguyen}, Section 6.3, we can then deduce that for $\tau$ small enough, we have $G_N\leq C_M$ for some $C_M>0$ independent from $N$, from which we conclude. Let $n\geq 1$. For each $p,q$ by \eqref{eq:recursfinal}, Proposition \ref{projbound} (in the case $n=1,\alpha=\alpha_0$) and Proposition \ref{genbound}  we have, for some constants $C_{U,k}>0$,
\begin{equation}\label{eq:gnpq}nG^{n,p,q}\leq (p+1)G^{n,p+1,q}+\sum_{1\leq j\leq n-1}\sum_{k=0}^q\sum_{l=0}^p G^{j,k,l}(\partial_{z_1}+\partial_{z_2})G^{n-j,q-k,p-l}+\sum_{k=1}^MC_{U,k} G^{n,p-k,q-2^{\Lambda-1}k}. \end{equation}
Notice that for $p=2q$ the term containing $G^{n,p+1,q}$ disappears. We claim the inequality
\begin{equation}\label{eq:gnpq2}G^{n,p,q}\leq \frac{C_p}{n} \sum_{p'=p}^{2q}\l \sum_{1\leq j\leq n-1}\sum_{k=0}^q \sum_{l=0}^{p'} G^{j,k,l}(\partial_{z_1}+\partial_{z_2})G^{n-j,q-k,p'-l}+\sum_{k=1}^M C_{U,k} G^{n,p'-k,q-2^{\Lambda-1}k}\r. \end{equation}
\begin{proof}
We proceed by reverse induction on $p$, starting from $p=2q$. In the base case $p=2q$, the term $G^{n,p+1,q}$ from \eqref{eq:gnpq} disappears, as we have set $\psi_{\alpha}^{n,p,q}=0$ for all $\alpha$ and for all $p>2q$ (see Section \ref{timedep}). Therefore, \eqref{eq:gnpq2} holds with $C_p=C_{2q}=1$.\\

Now, suppose \eqref{eq:gnpq2} holds for all $p'$ such that $p<p'\leq 2q$, where $p<2q$. Then
\begin{align*}
nG^{n,p,q}&\leq (p+1)C_{p+1}\sum_{p'=p+1}^{2q}\l \sum_{1\leq j\leq n-1}\sum_{k=0}^{q}\sum_{l=0}^{p'}G^{j,k,l}(\partial_{z_1}+\partial_{z_2})G^{n-j,q-k,p'-l}+\sum_{k=1}^M C_{U,k} G^{n,p'-k,q-2^{\Lambda-1}k}\r \\
&+\sum_{1\leq j\leq n-1}\sum_{k=0}^q\sum_{l=0}^pG^{j,k,l}(\partial_{z_1}+\partial_{z_2})G^{n-j,q-k,p-l}+\sum_{k=1}^MC_{U,k} G^{n,p-k,q-2^{\Lambda-1}k} \\
&\leq \l (p+1)C_{p+1}+1\r  \sum_{p'=p}^{2q}\l \sum_{1\leq j\leq n-1}\sum_{k=0}^q\sum_{l=0}^{p'}G^{j,l,l}(\partial_{z_1}+\partial_{z_2})G^{n-j,q-k,p'-l}+\sum_{k=1}^MC_{U,k} G^{n,p'-k,q-2^{\Lambda-1}k}\r.
\end{align*}
\end{proof}
We can now multiply both sides of \eqref{eq:gnpq2} by $t^p \nu^{q2^{-\Lambda}}$ and sum over $p=0$ to $p=2q$, and $q=0$ to $q=M$.  Note that, for any given positive real numbers $a_{p',q}$, we have
$$\sum_{p=0}^{2q}C_p\sum_{p'=p}^{2q}a_{p',q}=\sum_{p'=0}^{2q}a_{p',q}\sum_{p=0}^{p'}C_p\leq \l \sum_{p=0}^{2M}  C_p\r\l\sum_{p=0}^{2q}a_{p,q}\r.$$
Therefore, letting $C:=\sum_{p=0}^{2M} C_p$, we obtain for all $n\geq 2$,
\begin{align*}
    G^n&\leq \frac{C}{n}\sum_{j=1}^{n-1}\sum_{q=0}^M\sum_{p=0}^{2q}\sum_{l=0}^q\sum_{k=0}^p t^k \nu^{l2^{-\Lambda}}G^{j,k,l} t^{p-k}\nu^{(q-l)2^{-\Lambda}}(\partial_{z_1}+\partial_{z_2})G^{n-j,p-k,q-l}\\ &+C\sum_{q=0}^M\sum_{p=0}^{2q}\sum_{k=1}^M  t^k \nu^{k/2}C_{U,k}t^{p-k}\nu^{q2^{-\Lambda}-k/2}G^{n,p-k,q-2^{\Lambda-1}k}\\
    &\leq \frac{C}{n} \sum_{j=1}^{n-1}G^{j}(\partial_{z_1}+\partial_{z_2})G^{n-j}+ M\sum_{k=1}^M C_{U,k}t^k\nu^{2^{-\Lambda}k}\,G^n.
\end{align*}

For $\nu$ small enough the second sum above can be absorbed in the left-hand side, and the proof then carries on as in the time-independent case.

 In conclusion, we have that the series
 $$\sum_{n=0}^{\infty}\sum_{p,q=0}^M \tau^{n} t^p\nu^{q2^{-\Lambda}}G^{n,p,q} $$
 converges to some positive real number, and hence the solution defined by \eqref{eq:sol} converges in the space of smooth functions with finite generator functions. In particular, if $1/C_{\infty}$ is the radius of convergence of the series \eqref{eq:mainseries}, then  $G^n\leq 2C_{\infty}^n$. Hence the series \eqref{eq:sol} defines a solution $U_0$ to \eqref{eq:nsvort}, leaving an error defined by \eqref{eq:remainder}. We need to prove that the series defining $R$ also converges for $\tau$ small enough.    We have
 
\begin{align*}
    \gen_{\delta}\l \sum_{n=1}^{N}\tau^nR^n\r&\leq \sum_{n=1}^N \tau^n \sum_{1\leq j\leq n-1}\sum_{p_1,p_2\leq  2M}\sum_{q_1+q_2>M}^{q_1,q_2\leq M}(\nu^{2^{-\Lambda}(q_1+q_2)}t^{p_1+p_2})G^{j,p_1,q_1}G^{n-j,p_2,q_2}\\
    &+\sum_{n=\nu^{-1/4}}^N\tau^n \sum_{1\leq j\leq n-1}\sum_{p_1,p_2\leq  2M}\sum_{q_1,q_2\leq M}(\nu^{2^{-\Lambda}(q_1+q_2)}t^{p_1+p_2})G^{j,p_1,q_1}G^{n-j,p_2,q_2}\\
    &= \sum_{p_1,p_2\leq 2M}\sum_{q_1+q_2>M}^{q_1,q_2\leq M}(\nu^{2^{-\Lambda}(q_1+q_2)}t^{p_1+p_2})\sum_{j=1}^{N-1}\tau^jG^{j,p_1,q_1}\sum_{n=j+1}^N \tau^{n-j} G^{n-j,p_2,q_2}\\
    &+\sum_{p_1,p_2\leq 2M}\sum_{q_1,q_2\leq M}(\nu^{2^{-\Lambda}(q_1+q_2)}t^{p_1+p_2})\sum_{j=1}^{N-1}\tau^jG^{j,p_1,q_1}\sum_{n>j\land \nu^{-1/4}}^N \tau^{n-j} G^{n-j,p_2,q_2}\\ 
    &\leq\sum_{p_1,p_2\leq 2M}\sum_{q_1+q_2>M}^{q_1,q_2\leq M}(\nu^{2^{-\Lambda}(q_1+q_2)}t^{p_1+p_2}) G_{N}^{p_1,q_1}G_N^{p_2,q_2}\\
    &+\sum_{p_1\leq 2M}\sum_{q_1\leq M}\l \nu^{2^{-\Lambda}(q_1+q_2)}t^{p_1+p_2} \r G_N^{p_1,q_1}\tau^{\nu^{-1/4}} \sum_{n>\nu^{-1/4}}\sum_{p_2\leq 2M}\sum_{q_2\leq M}\tau^{n-\nu^{-1/4}} G^{n,p_2,q_2}\\ 
    &\leq C_M \l \nu^{2^{-\Lambda}}t^{2}\r^{q_1+q_2}+ C_M \tau^{\nu^{-1/4}}\\ 
    &\leq \tilde{C_M}\l \nu^{2^{-\Lambda}}\log \nu^{-2}\r^M+ C_M\tau^{\nu^{-1/4}}.
\end{align*}
Thus for $\tau$ small enough the series defining $R$ converges to a smooth function with $\gen_{\delta}R<\infty$. Moreover, for $M$ large enough, the quantity above is smaller than $C_{M,{N_s}}\nu^{N_s}$.
\subsection{Sobolev bounds} \label{sobolev}
So far, we have constructed a sequence of approximate solutions $\mbf{u}^{\nu}$ to \eqref{eq:ns2} with a small forcing $\mbf{f}^{\nu}$ and which are close to $U_0(y/\sqrt{\nu})$ at time $t=0$. Here, the 'size' is measured by the generator functions, which only control the $L^{\infty}$ norm. However, \autoref{main} requires these conditions in terms of the $H^s$ norms, where $s$ may be arbitrarily large.  Furthermore, when returning to our original variables by reverting the isotropic rescaling \eqref{eq:iso}, we have to keep in mind that the $H^s$ norms gain a factor of $\nu^{-s/2}$. In this section we will show how to control these norms.

Let $\alpha \in \mathbb{Z}$. Assume that $(c,\nu)$ is not an Orr-Sommerfeld eigenvalue, and let $G_{c,\nu}$ be the associated Green function of the $\orr_{c,\nu}$ operator. By Proposition A.1 from our paper \cite{prevpaper2}, we know that given $\psi \in C^{\infty}(\mathbb{R}_+)$ bounded with its derivatives, letting
$$ \phi(y):=\int_0^{\infty}G_{c,\nu}(x,y)\psi(x)\,\mathrm{d}x,$$
we have for all $k\in \mathbb{N}_{\geq 0}$ and for all $y\geq 0$,
\begin{equation}\label{eq:greenest0} |\phi^{(k)}(y)|\leq C_k\left[ \max_{0\leq j\leq k-4}|\psi^{(j)}(y)|+\max_{0\leq j\leq 3\land k}\left|\int_0^{\infty} \partial_y^j G_{c,\nu}(x,y)\psi(x)\,\mathrm{d}x\right|\right],
\end{equation}
where $C_k$ depends on the first $k-2\lor 0$ derivatives of $U_0$ and $\alpha$.

 Hence for all $k\in \mathbb{N}$,
 \begin{equation}\label{eq:greenest1}\left\|  \phi\right\|_{H^k}\leq \begin{cases} C_k\left( \|\psi\|_{H^{k-4}}+\|G_{c,\nu}\|_{H^3}\|\psi\|_{L^2}\right),& k\geq 4,\\ 
  \|G_{c,\nu}\|_{H^k}\|\psi\|_{L^2} & 0\leq k\leq 3.
 
 \end{cases}
 \end{equation}
Now let $c_{n,\alpha}$ be defined as in \eqref{eq:cdef}, with $|\alpha| \lesssim \nu^{-1/4}$ as per our usual assumption. For $k= 3$, by \eqref{eq:l1bound0} we have for all $n\geq 1$,
 $$\|G_{c_{n,\alpha},\nu}\|_{H^3(\mathbb{R}_+)}\leq \frac{C}{1+|\Im c_{n,\alpha}|}\l |\alpha|+ \mu \r \leq \frac{C'}{n} \nu^{-1/2},
$$
for some constants $C,C'>0$ independent from $n,\alpha,c,$ and $\nu$.

\begin{prop}\label{sobolevbounds} There exists a constant $C_k>0$ such that for all $n\in \mathbb{N},$ $n\geq 2$ we have
$$\|\mbf{u}^n\|_{H^k}+\|\omega^n\|_{H^k}\leq n^{\left\lfloor (k-3)/4\right\rfloor}\nu^{1/8-k/4-n/2}C_k^n. $$
 As a result, the series
 $\sum_{n=1}^{\infty}\tau^{n-1}\mbf{u}^n $
 converges in $H^k$ for all $\tau < \nu^{-1/2}C_k^{-1}$.
\end{prop}
\begin{proof}  
For all $n\geq 1$, let $G^n$ be defined as in Proposition \ref{genconv}, let $C_k$ be the constant defined in \eqref{eq:greenest0}, and for $k\geq 1$ define
$$A_k:=C_k\max\lll 2;(1+|\Im c_1|) \sup_{n\in \mathbb{N}}\|G_{c_{n,\alpha},\nu}\|_{H^3(\mathbb{R}_+)};\sup_{n\in \mathbb{N}_{\geq 2}}(G^n)^{1/(n-1)}\rrr,\qquad B_k:=\frac{\|\mbf{u}^1\|_{H^{k+2}}+\|\omega^1\|_{H^{k+2}}}{1+C_{L^{\infty}\hookrightarrow H^2}},$$
where $C_{L^{\infty}\hookrightarrow H^2}$ is the constant from the Sobolev embedding. Note that from the definition of $A_k$ and $B_k$ it follows that
 $$\|\mbf{u}^n\|_{L^{\infty}}\leq G^n \leq A_k^{n-1},\qquad \|\omega^n\|_{L^{\infty}}\leq  A_k^{n-1}B_k\leq n^{\left\lfloor (k-3)/4\right\rfloor}A_k^{n-1}B_k,\qquad \forall n,k\in\mathbb{N}.$$
Furthermore, we have
$$A_k \lesssim \nu^{-1/2},\qquad B_k \lesssim \nu^{-3/8-k/4}.$$
 
 We shall prove that for $\ell := \left \lfloor (k-3)/4\right\rfloor\geq 1$ we have $$\|\mbf{u}^n\|_{H^k}+\|\omega^n\|_{H^k}\leq A_k^{n-1} B_k n^{\ell}.$$ We run two inductions, first on $n\geq 1$ and then on $k\geq 0$. For $n=1$,  the claim follows from \eqref{eq:u1}. For $n>1$, let $G_{c_{n,\alpha},\nu}$ be the Green function of $\orr_{c_{n,\alpha},\nu}$, where $c_{n,\alpha}$ is defined by \eqref{eq:cdef}.

 Then from \eqref{eq:orreqn} and \eqref{eq:greenest1}, for $0\leq k\leq 3$ we have
 \begin{align*}
    \|\psi^n\|_{H^{k}}&\leq \frac{A_k}{n}\sum_{j=1}^{n-1}\l \|\mbf{u}^j\|_{L^{\infty}}\|\omega^{n-j}\|_{L^2}+\|\mbf{u}^j\|_{L^2}\|\omega^{n-j}\|_{L^{\infty}} \r\leq  \frac{A_k}{n}\sum_{j=1}^{n-1}A_k^{n-2}B_k\\
    &\leq A_k^{n-1}B_k.
 \end{align*}
Now we proceed by induction over $k\geq 4$. In this case $\left \lfloor (k-4-3)/4\right\rfloor=\ell -1\geq 0$, and so by \eqref{eq:greenest1} together with the case $k\leq 3$ and inductive assumption we have
\begin{align*}
    \|\psi^n\|_{H^k}&\leq \frac{A_k}{2}\sum_{j=1}^{n-1}\l \|\mbf{u}^j\|_{L^{\infty}}\|\omega^{n-j}\|_{H^{k-4}}+\|\mbf{u}^j\|_{H^{k-4}}\|\omega^{n-j}\|_{L^{\infty}} \r+ A_k^{n-1}B_k\\
    &\leq A_k^{n-1}B_k\l  \frac{1}{2}\sum_{j=1}^{n-1}\l (n-j)^{\ell-1}+ j^{\ell -1}\r +1\r\leq A_k^{n-1}B_k n^{\ell }.
\end{align*}
\end{proof}
Thus the solution $\mbf{u}^{\nu}$ we constructed in Section \ref{timedep} belongs to $H^s(\mathbb{T}\times \mathbb{R}_+)$, where $s\geq 0$ was fixed at the beginning of the Section, for all times $t\geq 0$ such that $\tau \leq \nu^{-1/2}C_s^{-1}$, i.e. for all $t \leq -(N_s+1/2)\sqrt{\nu}\frac{\log \nu-\log C_s}{\Im c_0}$.

As a result, so will the forcing $\mbf{f}=\nabla^{\perp}R$, where $R$ was defined in \eqref{eq:remainder}, which also shows that by taking $M$ large enough (depending on $s$) we have
$$\|\mbf{f}(t)\|_{H^s}\leq \nu^{N_s}\tau, $$
in the isotropically-rescaled variables, for $\nu>0$ small enough, and hence $\|\mbf{f}(t)\|_{H^s}\leq \nu^{N_0}$ for all $t \leq -N_s\sqrt{\nu}\frac{\log \nu}{\Im c_0}$ in the original variables. This validates the assumption \eqref{eq:forcings} from \autoref{main} .
\subsection{Conclusion}
Let us return to the original variables in which \autoref{main} is stated, by reverting the isotropic rescaling \eqref{eq:iso}. With a notation abuse, we still denote with $\mbf{u}^{\nu}$ the approximate solution constructed in Section \ref{timedep}, even if it is now expressed in terms of the original variables. Recall that for all $t\geq 0$, by the normalization condition \eqref{eq:normaliz} we have
$\|\mbf{u}^1(t)\|_{L^{\infty}}= 1$ for all $t\geq 0$,
where $\mbf{u}^1$ was introduced in \eqref{eq:u1}.
Thus, at $t=0$, since $\sum_{n=1}^{\infty}\tau^{n-1} \mbf{u}^n$ converges in $H^s$ we can approximate
$$\|\mbf{u}^{\nu}(0,x,y)-\mbf{U}_0(y/\sqrt{\nu})\|_{H^s}\leq \nu^{N_s}\|\psi^1\|_{H^{s+1}}\leq \nu^{N_s+7/8-3s/4}\leq  \nu^{N_0},$$
by definition of $N_s$. This establishes \eqref{eq:main0}.

It remains to determine the $L^{\infty}$ instability of $\mbf{u}^{\nu}$ at a time $\tilde{T}^{\nu}\to 0^+$ as $\nu\to 0^+$. Let $T_0>0$ and let $\tilde{T}^{\nu}:=\sqrt{\nu}\l -N_s\frac{\log \nu}{\Im c_{0}}-T_0\r$, so that $\tau(\tilde{T}^{\nu})=\nu^{N_s}e^{\Im c_{0} \tilde{T}^{\nu}/\sqrt{\nu}}=e^{-\Im c_{0} T_0}$. Since the series $\sum_{n=1}^{\infty}\tau^{n-1} \|{\mbf{u}}^n\|_{L^{\infty}}$ converges for $\tau < C_{\infty}$, we know that $\|\mbf{u}^n(t)\|_{L^{\infty}}\leq C:=\sup_{n\in \mathbb{N}}\|\mbf{u}^n\|_{L^{\infty}}<\infty$, for all $t\leq \tilde{T}^{\nu}$.  We can choose $T_0$ large enough so that $\tau^2/(1-\tau) \leq 1/(2C)$. Furthermore, for $\nu$ small enough, we will have $\tilde{T}^{\nu}>0$.\\

We then have
\begin{align*}\|\mbf{u}^{\nu}(\tilde{T}^{\nu})-U^{\nu}(\tilde{T}^{\nu})\|_{L^{\infty}}&\geq \tau(\tilde{T}^{\nu})\|\mbf{u}^1(\tilde{T}^{\nu})\|_{L^{\infty}}- \sum_{n=2}^{\infty}\tau^n (\tilde{T}^{\nu})\|\mbf{u}^n(\tilde{T}^{\nu})\|_{L^{\infty}}\\
&\geq e^{-\Im c_{0} T_0}\l 1-\sum_{n=2}^{\infty}\tau^{n-1}(\tilde{T}^{\nu})\|\mbf{u}^n(\tilde{T}^{\nu})\|_{L^{\infty}}\r \\
&\geq e^{-\Im c_0 T_0}\l 1- C\frac{\tau^2}{1-\tau}\r \\
&\geq \frac{1}{2}e^{-\Im c_{0} T_0}=:\sigma>0.
\end{align*}
 This completes the proof of \autoref{main}, where in the original variables the instability time is given by $\tilde{T}^{\nu}$.

\begin{appendix}

\section{Necessity of the assumption on the derivatives of the shear flow at the origin}
Let $U_0$ be the shear flow around which we have constructed the instability as per \autoref{main}.  Both in this paper and our previous paper, \cite{prevpaper}, we have made the rather strong assumption that all the derivatives of $U_0$ vanish at the origin: $U_0^{(k)}(0)=0$ for all $k\in\mathbb{N}$.

This assumption is necessary in order to ensure that the terms
${U^{\nu}(\sqrt{\nu}t,y)-U_0(y)}$, and thus the forcing $S\mbf{w}$ defined in \eqref{eq:sdef}, are bounded in $H^s$ (and vanish as $\nu \to 0^+$) for all $s\geq 0$. It corresponds to enforcing the compatibility condition of all orders for the corresponding heat equation \eqref{eq:shflows}. This is a sufficient and necessary for the solution to be regular up to the boundary, specifically up to the point $(t,y)=(0,0)$. In this case, the boundary condition and initial condition tell us that, for all $k\in \mathbb{N}$:
$$0=\left.\partial_t^k\l \partial_y U^{\nu}-\nu^{1/2-\gamma}U^{\nu}\r\right|_{(0,0)}=\left.\l\partial_y^{2k+1}U^{\nu}-\nu^{1/2-\gamma}\partial_y^{2k}U^{\nu}\r\right|_{(0,0)}=U_0^{(2k+1)}(0)-\nu^{1/2-\gamma}U_0^{(2k)}(0).$$
As this must be true for all $\nu>0$, we must require all derivatives of $U_0$ to vanish at $y=0$. In \cite{grenier3} and \cite{paddick}, the compatibility conditions do not involve the viscosity, so one only needs to require $U_0^{(2k)}(0)=0$ and $U_0^{(2k+1)}(0)=U_0^{(2k)}(0)$ respectively, for all $k\in\mathbb{N}$. These weaker assumptions have the advantage that they still admit analytic flows. For instance, $U_0(y)=ye^{-y^2}$ satisfies the no-slip compatibility conditions $U_0^{(2k)}(0)=0$ for all $k\in \mathbb{N}$, as well as the exponential decay property.

Let $\gamma>1/2$, and let $U_0$ be smooth and satisfying \eqref{eq:uscond}$_1$ as well as $U_0^{(2k)}(0)=0$ for all $k\in \mathbb{N}$, but $U_0'(0)\neq 0$ (e.g. $U_0(y)=ye^{-y^2}$). We can take the odd extension  of $U_0$ to $\mathbb{R}$, which will be smooth over $\mathbb{R}$ and which we will keep calling $U_0$. The solution $U^{\nu}$ to \eqref{eq:shflows} is given by
$$ U^{\nu}(t,y)=\int_{-\infty}^{+\infty}K(t,y-x)u_0^{\nu}(x)\,\mathrm{d}x,\qquad t\in [0,T],$$
where 
\begin{align*} u_0^{\nu}(y)&=U_0(y)+\chi_{(-\infty,0]}V^{\gamma}(y,\nu),\\
V^{\nu,\gamma}(y)&=2\int_0^{-y}e^{\nu^{1/2-\gamma}(x+y)}U_0'(x)\,\mathrm{d}x=2\int_{0}^{-y} e^{-\nu^{1/2-\gamma}x}U_0'(-x-y)\,\mathrm{d}x.\end{align*}
Note that $V^{\nu,\gamma}(y)$ decays rapidly as $y\to -\infty$, reflecting the rate of decay of $U_0'$. For instance, with $U_0(y)=ye^{-y^2}$ and $\gamma=1$ we can explicitly compute
$$V^{\nu,\gamma}(y)=\frac{1}{2\nu}\l -2e^{-y^2}\sqrt{\nu}\l -1+e^{y(y+\nu^{-1/2})}+2\sqrt{\nu}y\r
+\sqrt{\pi}e^{-\frac{1}{4\nu}+\frac{y}{\nu}}\l -\textup{Erf}\l \frac{1}{2\sqrt{\nu}}\r+\textup{Erf}\l \frac{1}{2\sqrt{\nu}}+y\r \r\r.$$
In general, for all $k\geq 1$ and $y\leq 0$, through integration by parts we obtain
\begin{equation}\label{eq:vvk}
    (V^{\nu,\gamma})^{(k)}(y)=2(-1)^k\int_0^{-y}e^{\nu^{1/2-\gamma}x}U_0^{(k+1)}(-x-y)\,\mathrm{d}x+2e^{\nu^{1/2-\gamma}y}\sum_{j=1}^k(-1)^{j}U_0^{(j)}(0)\nu^{(j-k)(\gamma-1/2)}.
\end{equation}
As a result, as $\nu \to 0^+$ we have
$$\|(V^{\nu,\gamma})^{(k)}\|_{L^2(-\infty,0)}\sim C_{U_0^{(k+1)}}\nu^{\gamma-1/2}+ 2\sum_{j=1}^k|U_0^{(j)}(0)|\nu^{(j-k+1/2)(\gamma-1/2)},$$
for some constants $C_{U_0^{(j)}}>0$, $j=1,\dots, k+1$ which only depend on $U_0^{(j)}$.

Therefore, for $k\geq 2$ we have
$$\lim_{\nu \to 0^+}\|(V^{\nu,\gamma})^{(k)}\|_{L^2(-\infty,0)}=0 \iff U_0^{(j)}(0)=0\;\forall j=1,\dots,k-1\iff (V^{\nu,\gamma})^{(k)}(0)=2(-1)^{k}U_0^{(k)}(0)\;\forall \nu>0.$$
Moreover, if the above condition does not hold for some $k\geq 2$, then $\lim_{\nu \to 0^+}\|(V^{\nu,\gamma})^{(k)}\|_{L^2}=\infty.$ 
On the other hand, we always have $\lim_{\nu \to 0^+}\|V^{\nu,\gamma}\|_{L^2}=0$ and $\lim_{\nu \to 0^+}\|(V^{\nu,\gamma})'\|_{L^2}=0,$ as well as $V^{\nu,\gamma}(0)=0$, $(V^{\nu,\gamma})'(0)=2U_0'(0)$. In our case, since $U_0'(0)\neq 0$, we conclude that $\lim_{\nu \to 0^+}\|(V^{\nu,\gamma})^{(k)}\|_{L^2}=\infty$ if and only if $k\geq 2$.

Now let $U^{0}$ be the solution to the heat equation with no-slip boundary condition and initial condition $U_0$. Thanks to the compatibility condition $U_0^{(2k)}(0)=0$ for all $k\in \mathbb{N}$, we know that
$$\lim_{\nu\to 0^+} \|U^{0}(t)-U_0\|_{L^{\infty}(0,T;H^s)}=0,\qquad \forall s\geq 0,t\geq 0.$$ 
Therefore, in order to study the behavior of $U^{\nu}-U_0$ as $\nu \to 0^+$ it is equivalent to consider
$$ \lim_{\nu\to 0^+}\|U^{\nu}(\sqrt{\nu}t,y)-U^{0}(\sqrt{\nu}t,y)\|_{L^{\infty}(0,T;H^s)}.$$
We have
$$U^{\nu}(\sqrt{\nu}t,y)-U^0(y,\sqrt{\nu} t)=K(\sqrt{\nu}t,y)\star_y (u_0^{\nu}(y)-U_0(y))=\int_{-\infty}^0K(\sqrt{\nu}t,y-x)V^{\nu,\gamma}(x)\,\mathrm{d}x.$$
Thus, for $k=0$ and $k=1$ we have
\begin{align*}\lim_{\nu \to 0^+}\|\partial_y^kU^{\nu}(\sqrt{\nu}t,y)-\partial_y^kU^0(\sqrt{\nu}t,y)\|_{L^2(0,\infty)}\leq \lim_{\nu \to 0^+}\|(V^{\nu,\gamma})^{(k)}\|_{L^2}= 0.\end{align*}
However, if we take the second derivative, then
\begin{equation}\label{eq:d2}
    \partial_y^2 (U^{\nu}-U^0)(\sqrt{\nu}t,y)=\int_{-\infty}^0K(\sqrt{\nu}t,y-x)(V^{\nu,\gamma})''(x)\,\mathrm{d}x-K(\sqrt{\nu}t,y)(V^{\nu,\gamma})'(0).
\end{equation}
In the above relation, the first term may be seen as the solution to the heat equation over $\mathbb{R}$ with initial condition $(V^{\nu,\gamma})''\chi_{(-\infty,0]}.$  As $\nu\to 0^+$ we have
\begin{align*}
    \|\partial_y^2 (U^{\nu}-U^0)(\sqrt{\nu}t,y)\|_{L^2(0,\infty)}\geq \left|\frac{1}{(8\pi \sqrt{\nu}t)^{1/4}}|U_0'(0)|-\|(V^{\nu,\gamma})''\|_{L^2}-o(1)\right|\geq \left| \frac{\nu^{1/8}}{t^{1/4}}-1-o(1)\right|\frac{2}{\nu^{1/4}}.
\end{align*}
This shows that the difference between the second derivatives of the slip and the no-slip solutions is unbounded in the energy norm as $\nu\to 0^+$ for any time $t>0$. Since the latter converges to the second derivative of the initial condition $U_0''$, we deduce that
$$\lim_{\nu \to 0^+}\|\partial_y^2 U^{\nu}(\sqrt{\nu}t,y)-U_0''(y)\|_{L^2(0,\infty)}=\infty.$$

More generally, if the compatibility conditions are not satisfied for some $k\in \mathbb{N}$, then the $(k+1)$-th spatial derivative of $U^{\nu}(t,y)$ (and thus of $U^{\nu}(t,y)-U_0(y)$) will be discontinuous at $(t,y)=(0,0)$. In fact, the energy of $\partial_yU^{\nu}-U_0'$, which is zero at $t=0$, will shoot up instantly to accomodate for the gap between the boundary and initial conditions. Therefore, $\lim_{\nu \to 0^+}\|U^{\nu}(\sqrt{\nu}t,y)-U_0(y)\|_{H^k}\neq 0$. This shows that our assumption that $U_0^{(k)}(0)=0$ for all $k\in \mathbb{N}$ is necessary if we want to follow the approximate construction used in \cite{greniernguyen}. However, it does not imply that it is a necessary assumption for the statement contained in \autoref{main}, as a different strategy for the proof could still eliminate this assumption.

\subsection{Numerical simulation}
Let $w^{\nu}(t,y):=U^{\nu}(\sqrt{\nu}t,y)-U_0(y)$. Then $w^{\nu}(t,y)$ satisfies the equations
\begin{equation}
    \begin{cases}\partial_t w^{\nu} = \partial_{yy}w^{\nu}+U_0'';\\
    w^{\nu}-\nu^{\gamma-1/2}\partial_yw^{\nu}+ U_0(0)-\nu^{\gamma-1/2}U_0'(0)=0 & y=0;\\
    w^{\nu}=0 & t=0.   
    \end{cases}
\end{equation}
Choosing $\gamma=1$, which is the physically relevant case, we find an approximate solution to the above problem by truncating the domain to $(t,x)\in [0,1]\times [0,5]$,  using MATLAB's pdepe solver. At $x=5$ we require a homogeneous Dirichlet boundary condition. We approximate the domain using quadratically spaced arrays to better capture the behavior near $(t,x)=(0,0)$. In the figures below, we plot the $L^2$ norms of derivatives of $w^{\nu}$ as a function of $t$ and $\sqrt{\nu}$.
\begin{figure}[h]
\centering
\caption{Case of $U_0(y)=ye^{-y^2}$.}
\begin{subfigure}{.5\textwidth}
\includegraphics[scale=0.6]{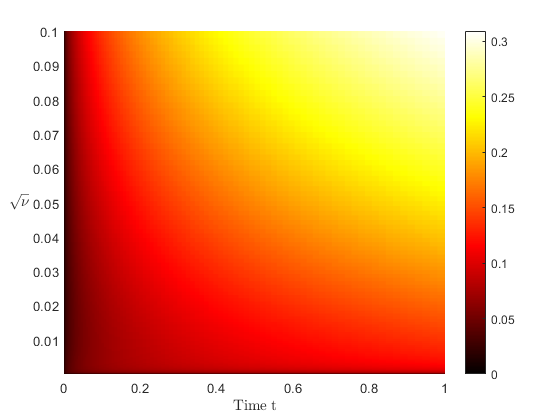}
\caption{$L^2$ norm of $\partial_y w^{\nu}(t,\cdot)$}
\label{fig:sub1}
\end{subfigure}%
\begin{subfigure}{.5\textwidth}
\includegraphics[scale=0.6]{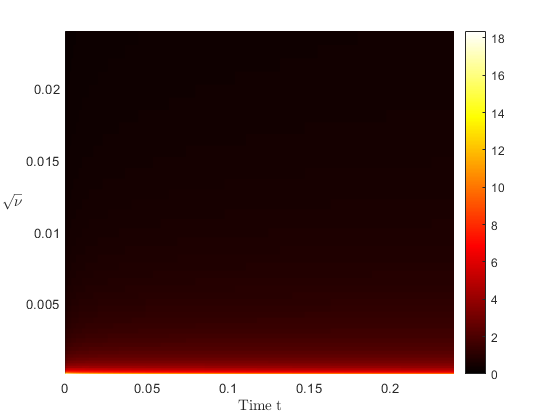}
\caption{$L^2$ norm of $\partial_y^2 w^{\nu}(t,\cdot)$}
\label{fig:sub2}
\end{subfigure}
\label{fig:1norm}
\end{figure}
In Figure \ref{fig:sub1}, taking $U_0(y)=ye^{-y^2}$, we can see that the $L^2$ norm of $\partial_y w$ converges to zero as $\nu \to 0$, regardless of the value of $t>0$. 
\begin{figure}[h]
\centering\caption{Case of $U_0(y)=e^{-1/y}$: $L^2$ norm of $\partial_y^{10} w^{\nu}(t,\cdot)$.}
\label{fig:2norm10}

\includegraphics[scale=0.6]{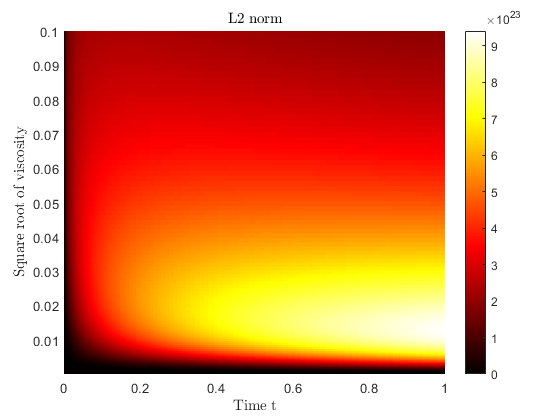}
\end{figure}
However, when we consider the second order derivative in  Figure \ref{fig:sub2}, we notice that even for small times $t$, convergence does not occur. This matches with our theoretical analysis, given that $U_0(y)=e^{-y^2}$ only satisfies the compatibility condition at even orders. Finally, if we switch to the flow $U_0(y)=e^{-1/y}$, which satisfies the compatibility conditions at all orders, then even for the $10$-th order derivative we can still visualize the convergence as in Figure \ref{fig:2norm10}.

\end{appendix}

\end{document}